\documentclass[10pt,final,reqno]{amsart}
\usepackage{amsfonts,latexsym,url,amsmath,amssymb,comment,enumitem}
\usepackage{upgreek}
\excludecomment{commentforus}
\usepackage{subfigure}
\usepackage{url,centernot}
\usepackage{graphicx,color}
\usepackage[color,notref,notcite]{showkeys}
\usepackage{bm}
\usepackage{stmaryrd}
\usepackage{textcomp}
\usepackage[ocgcolorlinks,allcolors={blue},breaklinks]{hyperref}

\DeclareMathAlphabet\mathbfcal{OMS}{cmsy}{b}{n}

\newcommand{\mathbi}[1]{{\boldsymbol #1}}
\def\N{\mathbb{N}}

\def\R{\mathbb{R}}

\def\err{\mathsf{err}}

\def\dist{{\rm dist}}

\def\dx{\,{\rm d}x}
\def\P{\mathbfcal{P}}

\def\<{\langle}
\def\>{\rangle}

\def\dsp{\displaystyle} 
\def\div{{\rm div}}

\def\norm#1#2{\Vert#1\Vert_{#2}}

\newcounter{cst}

\def\bt{\begin{theorem}}
\def\et{\end{theorem}}
\def\bl{\begin{lemma}}
\def\el{\end{lemma}}
\def\bc{\begin{corollary}}
\def\ec{\end{corollary}}
\def\bd{\begin{definition}}
\def\ed{\end{definition}}
\def\br{\begin{remark}}
\def\er{\end{remark}}

\newcommand{\disc}{{\mathcal D}}

\def \hessian{\mathcal H}
\def \hb{\hessian^{B}}
\def \hd{\hessian_\disc}
\def \hbd{\hessian_\disc^{B}}
\def \wdspace{H^B}
\def\cv{K}
\def \symd{\mathcal S_d}
\def \sym2{\mathcal S_2}
\def\cell{K}
\DeclareMathOperator*{\argmin}{argmin}

\def\P1{\mathbb{P}_1}

\def\stab{\mathfrak{S}}

\newcommand{\polyd}{{\mathcal T}}

\newcommand{\mesh}{{\mathcal M}}

\newcommand{\edge}{{\sigma}}

\newcommand{\edges}{{\mathcal F}}              
\newcommand{\edgescv}{{{\edges}_\cv}}  
  
\newcommand{\edgesext}{{{\edges}_{\rm ext}}} 
\newcommand{\edgesint}{{{\edges}_{\rm int}}}

\newcommand{\centers}{\mathcal{P}}
\newcommand{\x}{\mathbi{x}}
\newcommand{\z}{\mathbi{z}}

\newcommand{\y}{\mathbi{y}}

\newcommand{\centeredge}{\overline{\mathbi{x}}_\edge}

\def\dist{\mathrm{dist}}

\newcommand{\vertex}{{\mathsf{v}}}

\newcommand{\bu}{\overline{u}}

\newcommand{\be}{\begin{equation}}
\newcommand{\ee}{\end{equation}}

\renewcommand{\O}{\Omega}

\def\dr{\partial}

\renewcommand{\d}{{\:\rm d}}

\newcommand{\ba}{\begin{array}{llll}   }
\newcommand{\bac}{\begin{array}{c}}
\newcommand{\bari}{\begin{array}{r}}
\newcommand{\ea}{\end{array}}

\newcommand{\NORM}[1]{{\left\vert\kern-0.25ex\left\vert\kern-0.25ex\left\vert #1 
    \right\vert\kern-0.25ex\right\vert\kern-0.25ex\right\vert}}

\newcommand{\dm}{\disc_m}  
\def\hat#1{\widehat{#1}}

\def\assum#1{{\rm\textbf{(P#1)}}}

\newtheorem{theorem}{Theorem}[section]
\newtheorem{remark}[theorem]{Remark}

\newtheorem{lemma}[theorem]{Lemma} 
\newtheorem{definition}[theorem]{Definition}

\newtheorem{corollary}[theorem]{Corollary}

\numberwithin{equation}{section}

\def\XXint#1#2#3{{\setbox0=\hbox{$#1{#2#3}{\int}$ }
\vcenter{\hbox{$#2#3$ }}\kern-.6\wd0}}

\usepackage[normalem]{ulem}
\usepackage{color,colordvi}

\newcounter{cexp}
\catcode`\@=11
\def\terml#1{T_{\refstepcounter{cexp}\@bsphack
\protected@write\@auxout{}%
           {\string\newlabel{#1}{{\thecexp}{\thepage}}}\thecexp}}
\catcode`\@=12

\begin{document}
\title[The hessian discretisation method]{The Hessian discretisation method for fourth order linear elliptic equations}
\author{J\'er\^ome Droniou}
\address{School of Mathematical Sciences, Monash University, Clayton,
	Victoria 3800, Australia.
	\texttt{jerome.droniou@monash.edu}}
\author{Bishnu P. Lamichhane}
\address{School of Mathematical and Physical
	Sciences, University of Newcastle,
	University Drive, Callaghan, NSW 2308, Australia. \texttt{Bishnu.Lamichhane@newcastle.edu.au}}
\author{Devika Shylaja}
\address{IITB-Monash Research Academy, Indian Institute of Technology Bombay, Powai, Mumbai 400076, India.
	\texttt{devikas@math.iitb.ac.in}}
\maketitle

\begin{abstract}
In this paper, we propose a unified framework, the Hessian discretisation method (HDM), which is based on four discrete elements (called altogether a Hessian discretisation) and a few intrinsic indicators of accuracy, independent of the considered model. An error estimate is obtained, using only these intrinsic indicators, when the HDM framework is applied to linear fourth order problems. It is shown that HDM encompasses a large number of numerical methods for fourth order elliptic problems: finite element methods (conforming and non-conforming) as well as finite volume methods. We also use the HDM to design a novel method, based on conforming $\P1$ finite element space and gradient recovery operators. Results of numerical experiments are presented for this novel scheme and for a finite volume scheme.
\end{abstract}

\medskip

{\scriptsize
	
	\smallskip
	
}
\section{Introduction}
Fourth order elliptic partial differential equations arise in various applications, such as structural engineering, thin plate theories of elasticity, thin beams, biharmonic problems, the Stokes problem, image processing, etc. A large number of schemes, such as finite element (conforming, non-conforming) and finite volume methods, have been developed for the numerical approximation of these models. The purpose of this paper is to introduce a unified analysis framework, the Hessian discretisation method (HDM), that covers most of these schemes; by highlighting key abstract properties that ensure the scheme's convergence, the HDM also enables the design of novel schemes. We focus here on linear fourth order problem; non-linear models will be covered in a forthcoming paper.

\smallskip

The principle of the HDM, inspired by the Gradient Discretisation Method for 2nd order problems \cite{koala}, is to first select four discrete elements (a space and three reconstruction operators), altogether called a Hessian discretisation (HD). These elements are then substituted, in the weak formulation of the model, to the corresponding continuous space and operators, giving rise to a numerical scheme; this scheme is called a Hessian scheme (HS). A few indicators only, independent of the model and related to the coercivity, consistency and limit-conformity of the HD, are required to write error estimates in $L^2$, $H^1$ and $H^2$ norms for the corresponding HS. We show that schemes of the finite element and finite volume families fit into the HDM, with proper choices of HD, and we design a novel method based on the conforming $\P1$ space and a gradient recovery operator.

\smallskip

The finite element (FE) method is one of the most well-known tools for solving fourth-order elliptic boundary value problems. When conforming finite elements are used, the corresponding space must be a subspace of $H^2_0(\O)$. The corresponding strong continuity requirement of function and its derivatives makes it difficult to construct such a finite element, and leads to schemes with a large number of unknowns \cite{BS94,dou_percell_scott,ciarlet1978finite,percell_ctfem,powell_sabin}. It is known that to consider a conforming finite element space with $C^1$ continuity for a fourth-order problem, like the plate bending problem, a polynomial of degree at least 5 with 18 parameters (Bell's triangle) is required for a triangular element, and a bi-cubic polynomial with 16 parameters for a rectangular element (Bogner-Fox-Schmit rectangle) \cite{ciarlet1978finite}. The nonconforming finite element method relaxes the continuity requirement, which has a great impact on the resulting scheme. For the fourth order problem, two interesting nonconforming elements are the Adini rectangle and the Morley triangle \cite{ciarlet1978finite}. The finite element methods have been well-developed for the fourth order partial differential equation with variable\ constant coefficients, biharmonic problem and the bending problem, see \cite{BSBPK84,BPK85,ZX92,TS80,JL99,YLRAKL07,RF78,RERH10,BL_stab.mixedfem,PLPL75,NNBPK96}. We refer to \cite{biharmonicFV} and the reference therein for a discussion of finite volume methods for the biharmonic problem on general meshes. The interest of the method in \cite{biharmonicFV} is that it is easy to implement, computationally cheap
and requires only one unknown per cell. The analysis in \cite{biharmonicFV} is first based on meshes that respect an adequate orthogonality property, and then generalized to general polygonal meshes. In \cite{BL_FEM}, a finite element method for the biharmonic equation is presented; this method is based on gradient recovery operator, where the basis functions of the two involved spaces satisfy a condition of biorthogonality. The main idea is to use the gradient recovery operator to lift the non-differentiable, piecewise-constant gradient of $\P1$ finite element functions into the $\P1$ finite element space itself; the lifted functions are thus differentiable, and can be used to compute some kind of Hessian matrix of $\P1$ finite element functions. Ensuring the coercivity of the method in \cite{BL_FEM} on generic triangular/tetrahedral meshes however requires the addition of a stabilisation term. We also refer to \cite{CGZ16} for the application of the gradient recovery operator to fourth order eigenvalue problems.

\smallskip

We note that the interest of the HDM is that it extends the analysis beyond the setting of FE methods. It covers in particular situations where the second Strang lemma cannot be applied either because the continuous bilinear form cannot be extended to the space of discrete functions, and match there the discrete bilinear form, or even because the discrete space used in the scheme is not a space of functions (and the sum of the continuous and discrete spaces does not make sense).
 
\smallskip

The paper is organised as follows. In Section \ref{sec.modelproblem}, we introduce the model problem and  list some important examples of fourth order problems. We present the Hessian discretisation method in Section \ref{sec.HDM}, together with the error estimate established in this framework. In Section \ref{sec.grmethod}, we present a novel scheme based on the $\P1$ FE space and a gradient recovery designed using biorthogonal systems; this scheme does not require additional stabilisation terms, as the corresponding Hessian discretisation is built to already satisfy all  required coercivity properties. In Section \ref{sec.fvm}, we show that the finite volume method in \cite{biharmonicFV} is an HDM, and that the generic error estimate established in the HDM slightly improves the estimates found in \cite{biharmonicFV}, see Remark \ref{FVM.remark} below. Numerical results are presented to illustrate the theoretical convergence rate established in the HDM for the gradient recovery method and finite volume method in Section \ref{sec.example}. In Section \ref{sec.fem}, we show that some known schemes (conforming and non-conforming FE schemes) fit into the HDM. Finally, some technical results are gathered in an appendix.

\smallskip

\textbf{Notations}. A fourth order symmetric \emph{tensor} $P$ is a linear map $\symd(\R) \rightarrow \symd(\R)$, where $\symd(\R)$ is the set of symmetric matrices, $d$ is the dimension; $p_{ijkl}$ denote the indices of the fourth order tensor $P$ in the canonical basis of $\symd(\R)$. For simplicity, we follow the Einstein summation convention unless otherwise stated, i.e, if an index is repeated in a product, summation is implied over the repeated index. For  $\xi \in \symd(\R)$, using the definition of symmetric tensor, one has $P\xi \in \symd(\R)$ and $p_{ijkl}=p_{jikl}=p_{ijlk}$. The scalar product on $\symd(\R)$ is defined by $\xi:\phi=\xi_{ij}\phi_{ij}$. For a function $\xi: \O \rightarrow \symd(\R)$, denoting the Hessian matrix by $\hessian$ we set $\hessian : \xi = \partial _{ij}\xi_{ij}$. Finally, the transpose $P^{\tau}$ of $P$ is given by $P^{\tau} = (p_{klij})$, if $P = (p_{ijkl})$. Note that $P^{\tau}\xi : \phi = \xi : P \phi$.
The tensor product $a\otimes b$ of two vectors $a,b\in\R^d$ is the 2-tensor with coefficients $a_ib_j$. The Euclidean norm on $\R^d$ is denoted by $|{\cdot}|$, as is the induced norm on $\symd(\R)$. The Lebesgue measure of a measurable set $E\subset \R^d$ is denoted by $|E|$ (note that the nature of the argument of $|{\cdot}|$, a vector or a set, makes it clear if we talk about the Euclidean norm or the Lebesgue measure). The norm in $L^2(\O)$, $L^2(\O)^d$ for vector-valued functions, and $L^2(\O;\R^{d\times d})$ for matrix-valued functions, is denoted by $\|{\cdot}\|$.

\section{Model problem} \label{sec.modelproblem}

Let $\O \subset \R^d$ be a bounded domain  with boundary $\partial \Omega$ and consider the following fourth order model problem with clamped boundary conditions.
\begin{subequations} \label{model_problem}
	\begin{align}
& {\sum_{i,j,k,l=1}^{d}\partial_{kl}(a_{ijkl}\partial_{ij}\bu) = f  \quad \mbox{ in } \Omega, }\label{problem}\\
	 &\qquad \qquad \quad \bu=\frac{\partial \bu}{\partial n}= 0\quad\mbox{ on $\partial\O$}, \label{bc1}
	\end{align}
\end{subequations} 
where $ \x=(x_1,x_2,...,x_d) \in \O$,  $f \in L^2(\O)$, $n$ is the unit outer normal to $\O$ and the coefficients $a_{ijkl}$ are measurable bounded functions which satisfy the conditions $ a_{ijkl}=a_{jikl}=a_{ijlk} =a_{klij}$ for $i,j,k,l=1,\cdots,d .$
For all $\xi, \phi \in \symd(\R)$,\, we assume the existence of a fourth order tensor $B$ such that $A\xi:\phi =B\xi:B\phi,$ where $A$ is the four-tensor with indices $a_{ijkl}$. We notice that $B\xi:B\phi = B^{\tau}B\xi : \phi$, so that $A = B^{\tau}B$.

\medskip

Setting 
\[
\begin{aligned}
V=H^2_{0}(\O)={}&\left\{ v \in H^2(\O); v=\frac{\partial v}{\partial n} = 0 \mbox{ on } \partial\O \right\}\\
={}&\left\{ v \in H^2(\O); v=|\nabla v| = 0 \mbox{ on } \partial\O \right\},
\end{aligned}
\]
the weak formulation of \eqref{model_problem} is
\be\label{weak}
	\mbox{Find  $\bu \in V$ such that $\forall v \in V$,}
\quad \int_\O \hb \bu:\hb v \d\x=\int_\O fv\d\x,
\ee
where $\hb v=B\hessian v$. Note that $\int_\O \hb \bu:\hb v \d\x=\int_\O A\hessian \bu:\hessian v\d\x$, since $A=B^\tau B$.
We assume in the following that $B$ is constant over $\O$, and that the following coercivity property holds:
\be \label{coer:B}
\exists \varrho>0\mbox{ such that } \norm{\hb v}{}\ge \varrho\norm{v}{H^2(\Omega)},\;\forall v\in H^2_0(\Omega).
\ee
Hence, the weak formulation \eqref{weak} has a unique solution by the Lax--Milgram lemma.

\begin{remark} Adapting the analysis of Section \ref{sec.HDM} to $B$ dependent on $\x\in \O$ is easy,
provided that the entries of $B$ belong to $W^{2,\infty}(\Omega)$.
\end{remark}

\subsection{Examples}\label{sec:examples}

Let us examine two specific examples of the abstract problem \eqref{model_problem}. 

\smallskip

\subsubsection{Biharmonic problem}\label{sec:bihar}
The biharmonic problem is
\be
\label{biharmonic}
\Delta^2u=f \mbox{ in } \O,\qquad u=\frac{\partial u}{\partial n}=0 \mbox{ in }\partial \O.
\ee
The weak formulation of this model is given by \eqref{weak} provided that $B$ is chosen to satisfy
\[
\int_{\O}^{} \hb u : \hb v\d\x=\int_{\O}^{} \Delta u\Delta v\d\x.
\]
One possible choice of $B$ is therefore to set $B\xi=\frac{\rm{tr}(\xi)}{\sqrt{d}}{\rm Id}$ for $\xi \in \symd(\R)$ (where ${\rm Id}$ is the identity matrix), in which case $\hb=\Delta$. Since $\int_{\O}^{} \Delta u\Delta v\d\x=\int_\O \hessian u:\hessian v\d\x$, another possibility is to set $B$ the identity tensor ($B\xi=\xi$), in which case $\hb=\hessian$. By the Poincar\'e inequality, both choices satisfy \eqref{coer:B}.

\smallskip

\subsubsection{Plate problem}
The clamped plate problem \cite[Chapter 6]{ciarlet1978finite} corresponds to \eqref{weak} with $d=2$ and left-hand side
\begin{equation}\label{plate.lhs}
\int_{\O}^{} \Delta u\Delta v +(1-\gamma)(2\partial_{12}u \partial_{12}v-\partial_{11}u \partial_{22}v-\partial_{22}u \partial_{11}v)\d\x.
\end{equation}
Here, the constant $\gamma$ lies in the interval $(0, \frac{1}{2})$. We notice that \eqref{plate.lhs} is equal to $\int_\O A\hessian u:\hessian v\d\x$, where the tensor $A$ has non-zero indices $a_{1111}=1$, $a_{2222}=1$, $a_{1212}=(1-\gamma)$, $a_{2121}=(1-\gamma)$, $a_{1122}=\gamma$ and $a_{2211}=\gamma$. Its `square root' can be defined as the tensor $B$ with non-zero indices $b_{1111} = b_{2222} = \sqrt{\frac{1+\sqrt{1-\gamma^2}}{2}}$, $b_{1122} = b_{2211} = \sqrt{\frac{1-\sqrt{1-\gamma^2}}{2}}$ and $b_{1212}= b_{2121}=\sqrt{1-\gamma}$. It can be checked that \eqref{coer:B} holds since, for some $\varrho>0$, $A\mathbi{\xi}:\mathbi{\xi}\ge \varrho^2 |\mathbi{\xi}|^2$ for all $\mathbi{\xi}\in\symd(\R)$.

\section{The Hessian discretisation method}\label{sec.HDM}
We present here the Hessian discretisation method, and list the properties that are required for the convergence analysis of the Hessian scheme. The error estimate is stated at the end of the section.

\begin{definition}[$B\textendash$Hessian discretisation]\label{GD}
A $B\textendash$Hessian discretisation for clamped boundary conditions is a quadruplet $\disc=(X_{\disc,0},\Pi_\disc,\nabla_\disc,\hbd)$ such that
\begin{itemize}
\item $X_{\disc,0}$ is a finite-dimensional space encoding the unknowns of the method,
\item $\Pi_\disc:X_{\disc,0}  \rightarrow L^2(\O)$ is a linear mapping that reconstructs a function from the unknowns,
\item $\nabla_\disc:X_{\disc,0}  \rightarrow L^2(\O)^d$ is a linear mapping that reconstructs a gradient from the unknowns,
\item $\hbd:X_{\disc,0}  \rightarrow L^2(\O;\R^{d\times d})$ is a linear mapping that reconstructs a discrete version of $\hb(=B\hessian)$ from the unknowns. It must be chosen such that $\norm{\cdot}{\disc}:=\norm{\hbd \cdot}{}$ is a norm on  $X_{\disc,0}.$
\end{itemize}
\end{definition}

\begin{remark}[Dependence of the Hessian discretisation on $B$]
In the (2nd order) gradient discretisation method, the definition of a gradient discretisation is independent of the differential operator. Here, our definition of Hessian discretisation depends on $B$, that appears in the differential operator. This is justified by the fact that some methods (such as the one presented in Section \ref{sec.fvm}) are not built on an approximation of the entire Hessian of the functions, but only on some of their derivatives (such as the Laplacian of the functions). Although it might be possible to enrich these methods by adding approximations of the `missing' second order derivatives (as done in \cite{fvca8-agdm} in the context of the GDM), it does not seem to be the most natural way to proceed, and it leads to additional technicality in the analysis. Making the definition of HD dependent on the considered model through $B$ enables us to more naturally embed some known methods into the HDM.

Note however that a number of FE methods provide approximations of the entire Hessian of the functions (see Sections \ref{sec.grmethod} and \ref{sec.fem}). For those methods, a $B$-Hessian discretisation is built from an ${\rm Id}$-Hessian discretisation (that is independent of the model) by setting $\hbd=B\hessian_\disc^{\rm Id}$.

\end{remark}

If $\disc=(X_{\disc,0},\Pi_\disc,\nabla_\disc,\hbd)$ is a $B\textendash$Hessian discretisation, the corresponding scheme for \eqref{model_problem}, called Hessian scheme (HS), is given by
\be\label{base.HS}
\begin{aligned}
&\mbox{Find $u_\disc\in X_{\disc,0}$ such that for any $v_\disc\in X_{\disc,0}$,}\\
&\int_\O \hbd u_\disc:\hbd v_\disc\d\x=\int_\O f\Pi_\disc v_\disc\d\x.
\end{aligned}
\ee
This HS is obtained by replacing, in the weak formulation \eqref{weak}, the continuous space $V$ by $X_{\disc,0}$, and by using the reconstructions $\Pi_\disc$ and $\hbd$ in lieu of the function and its Hessian.

We will show that the accuracy of the HS can be evaluated using only three measures, all intrinsic to the Hessian discretisation. The first one is a constant, $C_\disc^B$, which controls the norm of the linear mappings $\Pi_\disc$ and $\nabla_\disc$.
 \be\label{def.CD}
 C_\disc^B = \max_{w\in X_{\disc,0}\backslash\{0\}} \left(\frac{\norm{\Pi_\disc w}{}}{\norm{\hbd w}{}},
 \frac{\norm{\nabla_\disc w}{}}{\norm{\hbd w}{}}\right).
 \ee
The second measure of accuracy is the interpolation error $S_\disc^B$ defined by
 \be\label{def.SD}
 \begin{aligned}
 	&\forall  \,\varphi\in H^2_0(\O)\,,\\
 	&S_\disc^B(\varphi)=\min_{w\in X_{\disc,0}}\Big(\norm{\Pi_\disc w-\varphi}{}
 +\norm{\nabla_\disc w-\nabla\varphi}{}+\norm{\hbd w-\hb \varphi}{}\Big).
 \end{aligned}
 \ee
Finally, the third quantity is a measure of limit-conformity of the HD, that is, how well a discrete integration-by-parts formula is verified by the discrete operators:
 \be\label{def.WD}
 \begin{aligned}
 	&\forall \, \xi \in \wdspace(\O):=\{\zeta\in L^2(\O)^{d \times d}\,;\,\hessian:B^{\tau}B\zeta \in L^2(\O) \}\,,\\
 	&W_\disc^B(\xi)=\max_{w\in X_{\disc,0}\backslash\{0\}}
 	\frac{1}{\norm{\hbd w}{}}\Bigg|\int_\O \Big((\hessian:B^{\tau}B\xi)\Pi_\disc w - B\xi:\hbd w \Big)\d\x \Bigg|.
 \end{aligned}
 \ee
Note that if $\xi \in \wdspace(\O)$ and $\phi\in H^2_0(\O)$, integration-by-parts show that $\int_{\O}^{}(\hessian:B^{\tau}B\xi)\phi = \int_{\O}^{}B\xi:\hb \phi$. Hence, the quantity in the right-hand side of \eqref{def.WD} measures a defect of discrete integration-by-parts between $\Pi_\disc$ and $\hbd$.

Closely associated to the three measures above are the notions of coercivity, consistency and limit-conformity of a sequence of Hessian discretisations.

 \begin{definition}[Coercivity, consistency and limit-conformity]
 Let $(\disc_m)_{m \in \N}$ be a sequence of $B\textendash$Hessian discretisations in the sense of Definition \ref{GD}. We say that
\begin{enumerate}
\item $(\disc_m)_{m\in\N}$ is \emph{coercive} if there exists $C_P \in \R^+$ such that $C_{\disc_{m}}^B \leq C_P$ for all $m \in \N$.
\item $(\disc_m)_{m\in\N}$ is \emph{consistent}, if 
 \be\label{def:cons}
\forall \varphi\in H^2_0(\O)\,,
 \lim_{m \rightarrow \infty} S_{\disc_m}^B(\varphi)=0.
\ee
\item $(\disc_m)_{m\in\N}$ is \emph{limit-conforming}, if 
\be\label{def:lc}
\forall \xi\in \wdspace(\O)\,,
\lim_{m \rightarrow \infty} W^B_{\disc_m}(\xi)=0.
\ee
\end{enumerate}
 \end{definition}

\begin{remark}\label{rem:dens} As for the (2nd order) gradient discretisation method, see \cite[Lemmas 2.16 and 2.17]{koala}, it is easily proved that, for coercive sequences of HDs, the consistency and
limit-conformity properties \eqref{def:cons} and \eqref{def:lc} only need to be tested for functions in dense subsets of $H^2_0(\O)$ and $H^B(\O)$, respectively.
\end{remark}

\begin{remark}\label{B_I_GD}
If $B = {\rm Id}$, we write $\hd$ (resp. $C_\disc$, $	S_\disc$ and $W_\disc$) instead of $\hessian_{\disc}^{{\rm Id}}$ (resp. $C_\disc^{\rm  Id}$, $	S_\disc^{\rm  Id}$ and $W_\disc^{\rm  Id}$).
\end{remark}

We can now state our main theorem giving the error estimates.

\begin{theorem}[Error estimate for Hessian schemes]\label{error} Under Assumption \eqref{coer:B}, let $\bu$ be the solution to \eqref{weak}. Let $\disc$ be a $B\textendash$Hessian discretisation and $u_\disc$ be the solution to the corresponding Hessian scheme \eqref{base.HS}. Then we have the following error estimates:
\begin{align}
\norm{\Pi_\disc u_\disc -\bu}{}& \le C_\disc W_\disc^B(\hessian \bu) +(C_\disc+1) S_\disc^B(\bu), \label{Pi_error} \\
\norm{\nabla_\disc u_\disc -\nabla\bu}{}
& \le C_\disc W_\disc^B(\hessian \bu) +(C_\disc+1) S_\disc^B(\bu), \label{nabla_error} \\
\norm{\hbd u_\disc-\hb \bu}{}
& \le W_\disc^B(\hessian \bu) +2S_\disc^B(\bu). \label{Hessian_error} 
\end{align}	
 (Note that $\hessian \bu \in \wdspace(\O)$ because $\hessian \bu \in  L^2(\O)^{d \times d} \mbox{ and }\hessian:B^{\tau}B \hessian \bu = \hessian : A \hessian \bu = f \in L^2(\O)$.)
\end{theorem}

The following convergence result is a trivial consequence of the error estimates above.

\begin{corollary}[Convergence]
Let $(\disc_m)_{m \in \N}$ be a sequence of $B\textendash$Hessian discretisations that is coercive, consistent and limit-conforming. Then, as $m \rightarrow \infty$, $\Pi_{\dm} u_{\dm} \rightarrow \bu$ in $L^2(\O)$, $\nabla_{\dm} u_{\dm} \rightarrow \nabla\bu$ in $L^2(\O)^d$ and $\hb_{\dm} u_{\dm} \rightarrow \hb \bu$ in $L^2(\O)^{d \times d}$.
\end{corollary}

Let us now prove Theorem \ref{error}.

\begin{proof}[Proof of Theorem \ref{error}]
For all $v_\disc \in X_{\disc,0}$, the equation \eqref{problem} taken in the sense of distributions shows that $f=\hessian:A\hessian\bu$, and thus, by the Hessian scheme \eqref{base.HS},
	 \begin{align*}
	 \int_\O \hbd u_\disc:\hbd v_\disc\d\x&=\int_{\O}f \Pi_\disc v_\disc\d\x 	
	 =\int_{\O}(\hessian:B^{\tau}B\hessian \bu)\Pi_\disc v_\disc\d\x. 
	 \end{align*}
Using the definition of $W_\disc^B$, we infer
\be \int_\O \Big(\hessian^B \bu-\hbd u_\disc\Big):\hbd v_\disc\d\x \le W_\disc^B(\hessian \bu)\norm{\hbd v_\disc}{}. \label{wdineq}
\ee
Define the interpolant $P_\disc : H^2_0(\O) \rightarrow X_{\disc,0}$ by
\[
P_\disc \bu= \argmin_{w\in X_{\disc,0}}\Big(\norm{\Pi_\disc w-\bu}{}
+\norm{\nabla_\disc w-\nabla\bu}{}	
+\norm{\hbd w-\hb \bu}{}\Big)
\]
and notice that
\be\label{approx.PD}
\norm{\Pi_\disc P_\disc \bu-\bu}{}+\norm{\nabla_\disc P_\disc \bu-\nabla\bu}{}
+\norm{\hbd P_\disc \bu-\hb\bu}{}\le S_\disc^B(\bu).
\ee
Introducing the term $\hb \bu$ and using \eqref{wdineq}, we obtain
\begin{align*}
\int_\O \Big(\hbd{}& P_\disc \bu-\hbd u_\disc\Big):\hbd v_\disc\d\x \\
={}&\int_\O \Big(\hb \bu-\hbd u_\disc\Big):\hbd v_\disc\d\x+\int_\O \Big(\hbd P_\disc \bu-\hb \bu\Big):\hbd v_\disc\d\x  \\
\le{}& W_\disc^B(\hessian \bu)\norm{\hbd v_\disc}{} +\norm{\hbd P_\disc \bu-\hb \bu}{}
\norm{\hbd v_\disc}. 
\end{align*}
Choosing $v_\disc=P_\disc \bu-u_\disc$, we get
\begin{align*}
\norm{\hbd( P_\disc \bu-u_\disc)}{}^2  \le{}& W_\disc^B(\hessian \bu)\norm{\hbd (P_\disc \bu-u_\disc)}{}\\
& + \norm{\hbd P_\disc \bu-\hb \bu}{}\norm{\hbd (P_\disc \bu-u_\disc)}{}. \nonumber
\end{align*}
Thus, by \eqref{approx.PD},
\be\norm{\hbd P_\disc \bu-\hbd u_\disc}{} \le W_\disc^B(\hessian \bu) +S_\disc^B(\bu). \label{HdPdineq2}
\ee
A use of triangle inequality, \eqref{approx.PD} and \eqref{HdPdineq2} yields
\begin{align}
\norm{\hbd u_\disc-\hb \bu }{}& \le \norm{\hbd u_\disc-\hbd P_\disc\bu}{} + \norm{\hbd P_\disc \bu-\hb \bu}{} \nonumber \\
& \le W_\disc^B(\hessian \bu) +2S_\disc^B(\bu), \nonumber
\end{align}
which is \eqref{Hessian_error}. Using the definition of $C_\disc$, and \eqref{approx.PD} and \eqref{HdPdineq2}, we obtain
\begin{align}
\norm{\Pi_\disc u_\disc -\bu}{}&\le \norm{\Pi_\disc u_\disc -\Pi_\disc P_\disc \bu}{} + \norm{\Pi_\disc P_\disc \bu -\bu}{} \nonumber \\
& \le C_\disc\norm{\hbd P_\disc \bu-\hbd u_\disc}{} +S_\disc^B(\bu) \nonumber \\
& \le C_\disc W_\disc^B(\hessian \bu) +(C_\disc+1) S_\disc^B(\bu). \nonumber
\end{align}
Hence, \eqref{Pi_error} is established, and \eqref{nabla_error} follows in a similar way.
\end{proof}

We now aim to present particular HDMs. The first (in Section \ref{sec.grmethod}) is a novel scheme based on gradient recovery operators, and a particular cheap construction of these operators using biorthogonal basis. Then, we show that a finite volume method (in Section \ref{sec.fvm}) and known finite element methods (in Section \ref{sec.fem}) fit into the HDM. Let us first set some notations related to meshes.

\begin{definition}[Polytopal mesh {\cite[Definition 7.2]{koala}}]\label{def:polymesh}~
	Let $\Omega$ be a bounded polytopal open subset of $\R^d$ ($d\ge 1$). A polytopal mesh of $\O$ is $\polyd = (\mesh,\edges,\centers)$, where:
	\begin{enumerate}
		\item $\mesh$ is a finite family of non empty connected polytopal open disjoint subsets of $\O$ (the cells) such that $\overline{\O}= \dsp{\cup_{\cell \in \mesh} \overline{\cell}}$. For any $\cell\in\mesh$, $|\cell|>0$ is the measure of $\cell$, $h_\cell$ denotes the diameter of $\cell$, $\overline{\x}_\cell$ is the center of mass of $K$, and $n_K$ is the outer unit normal to $K$.
		
		\item $\edges$ is a finite family of disjoint subsets of $\overline{\O}$ (the edges of the mesh in 2D, the faces in 3D), such that any $\edge\in\edges$ is a non empty open subset of a hyperplane of $\R^d$ and $\edge\subset \overline{\O}$. Assume that for all $\cell \in \mesh$ there exists  a subset $\edgescv$ of $\edges$ such that the boundary of $\cell$ is ${\bigcup_{\edge \in \edgescv}} \overline{\edge}$. We then set $\mesh_\edge = \{\cell\in\mesh\,;\,\edge\in\edgescv\}$ and assume that, for all $\edge\in\edges$, $\mesh_\edge$ has exactly one element and $\edge\subset\partial\O$, or $\mesh_\edge$ has two elements and $\edge\subset\O$. Let $\edgesint$ be the set of all interior faces, i.e. $\edge\in\edges$ such that $\edge\subset \O$, and $\edgesext$ the set of boundary faces, i.e. $\edge\in\edges$ such that $\edge\subset \dr\O$. The $(d-1)$-dimensional measure of $\edge\in\edges$ is $|\edge|$, and its centre of mass is $\centeredge$.
		
		\item $\centers = (\x_\cell)_{\cell \in \mesh}$ is a family of points of $\O$ indexed by $\mesh$ and such that, for all  $\cell\in\mesh$,  $\x_\cell\in \cell$.
		Assume that any cell $\cell\in\mesh$ is strictly $\x_\cell$-star-shaped, meaning that 
		if $\x\in\overline{\cell}$ then the line segment $[\x_\cell,\x)$ is included in $\cell$.
		
	\end{enumerate}
The diameter of such a polytopal mesh is $h=\max_{K\in\mesh}h_K$. 
\end{definition}

\section{Method based on Gradient Recovery Operators}\label{sec.grmethod}

\subsection{General setting}

Let $V_h$ be an $H^1_0$-conforming finite element space with underlying mesh
$\mesh=\mesh_h$. We assume that $V_h$ contains the piecewise linear functions,
and that $\mesh_h$ satisfies usual regularity assumptions, namely, denoting by
$\rho_\cell=\max\{r>0\,;\,B(\overline{\x}_\cell,r)\subset \cell\}$
the maximal radius of balls centred at $\overline{\x}_\cell$ and included in $\cell$, we
assume that there exists $\eta>0$ (independent of $h$) such that
\be\label{reg:mesh}
\forall \cell\in\mesh\,,\; \eta\ge \frac{h_K}{\rho_\cell}.
\ee
The gradient $\nabla u$ of $u\in V_h$ is well defined, but its second derivative $\nabla \nabla u$ 
is not. In order to compute some sort of second derivatives, consider a projector
$Q_h: L^2(\O) \rightarrow V_h$, which is extended to $L^2(\O)^d$ component-wise.
Then $\nabla u$ can be projected onto $ V^d_h$, and the resulting function $Q_h \nabla u\in V_h^d$
is differentiable. We can then consider $\nabla(Q_h\nabla u)$ as a sort of Hessian of $u$.
However, it not necessarily clear, for some interesting choices of practically computable $Q_h$ (see Section \ref{sec:biorth}), that this reconstructed
Hessian has proper coercivity properties. We therefore also consider a function $\stab_h$ whose
role is to stabilise this reconstructed Hessian.

Let $(V_h,Q_h,I_h,\stab_h)$ be a quadruplet of a finite element space $V_h\subset H^1_0(\O)$,
a reconstruction operator $Q_h:L^2(\O)\to V_h$ that is a projector onto $V_h$ (that is, $Q_h={\rm Id}$ on $V_h$), an interpolant $I_h:{H^2_0(\O)}\to V_h$
 and a stabilisation function $\stab_h\in L^\infty(\O)^d$ such that,
with constants $C$ not depending on $h$,
\begin{itemize}
	\item[\assum{0}] [\emph{Strucure of $V_h$ and $I_h$}] The inverse estimate $\norm{\nabla z}{}\le Ch^{-1}\norm{z}{}$ holds for all $z \in V_h$
	and, for $\varphi \in H^2_0(\O)$, we have 
	  $\norm{\nabla I_h \varphi-\nabla \varphi}{} \le Ch\norm{\varphi}{H^2(\O)}$. 
	\item[\assum{1}] [\emph{Stability of $Q_h$}] For $\phi \in L^2(\O)$, we have $\norm{Q_h \phi}{}
	\le C \norm{\phi}{}.$ 
	\item[\assum{2}] [\emph{$Q_h\nabla I_h$ approximates $\nabla$}] 
	For some space $W$ densely embedded in $H^3(\Omega) \cap H^2_0(\Omega)$ and for all $\psi\in W$, we have $\norm{Q_h\nabla I_h \psi-\nabla \psi}{} \le Ch^2\norm{\psi}{W}$.
	\item[\assum{3}] [\emph{$H^1$ approximation property of $Q_h$}] For $w \in {H^2(\O)\cap H^1_0(\O)}$, we have
	$\norm{\nabla Q_hw-\nabla w}{} \le Ch\norm{w}{H^2(\O)}$.
	\item[\assum{4}] [\emph{Asymptotic density of $[(Q_h\nabla-\nabla)(V_h)]^\bot$}] Setting $N_h=[(Q_h\nabla-\nabla)(V_h)]^\bot$, where the orthogonality is considered for the $L^2(\O)^d$-inner product, the following approximation property holds:
\[
\inf_{\mu_h\in N_h}\norm{\mu_h-\varphi}{}\le Ch\norm{\varphi}{H^1(\O)^d},\; \forall\varphi \in H^1(\O)^d,
\]
	\item[\assum{5}] [\emph{Stabilisation function}] $1\le |\stab_h| \le C$ and, for all $\cell \in \mesh$, 
	denoting by $V_h(\cell)=\{v_{|\cell}\,;\,v \in V_h\,,\; \cell \in \mesh\}$ the local FE space,
	$$\left[\stab_{h|K}\otimes (Q_h\nabla-\nabla)(V_h(\cell))\right] \perp \nabla V_h(\cell)^d,$$
where the orthogonality is understood in $L^2(K)^{d\times d}$ with the inner product
induced by ``$:$''.
\end{itemize}

\begin{remark} A classical operator $Q_h$ that satisfies these assumptions, for standard
FE spaces $V_h$, is the $L^2$-orthogonal projector on $V_h$. This operator is however non-local
and complicated to compute. We present in Section \ref{sec:biorth} a much more efficient construction
of $Q_h$, local and based on biorthogonal bases.
\end{remark}

To construct an HD based on such a quadruplet, we assume the following stronger form of \eqref{coer:B}:
\begin{equation}\label{coer:Bs}
\exists C_B>0\,:\, |B\mathbi{\xi}|\ge  C_B|\mathbi{\xi}|\,,\quad\forall \mathbi{\xi}\in\symd(\R).
\end{equation}

\begin{definition}[$B\textendash$Hessian discretisation using gradient recovery]\label{def:HDM:GR}
Under Assumption \eqref{coer:Bs}, the $B$-Hes\-sian discretisation based on a quadruplet $(V_h,Q_h,I_h,\stab_h)$ satisfying \assum{0}--\assum{5}
is defined by: $X_{\disc,0}=V_h$ and, for $u\in X_{\disc,0}$,
\[
\Pi_\disc u=u\,,\;\nabla_\disc u = Q_h\nabla u\mbox{ and }\hbd u=B\left[\nabla (Q_h \nabla u)+\stab_h\otimes (Q_h\nabla u-\nabla u)\right].
\]
\end{definition}

The next theorem gives an estimate on the accuracy measures $C_\disc^B$, $S_\disc^B$ and $W_\disc^B$ associated with an HD $\disc$ using gradient recovery.
Incidentally, the estimate on $C_\disc^B$ also establishes that 
$\norm{\hbd\cdot}{}$ is a norm on $X_{\disc,0}$.

\begin{theorem}[Estimates for Hessian discretisations based on gradient recovery]\label{thm:HD.coerciveB}~\\
	Let $\disc$ be a $B\textendash$Hessian discretisation in the sense of Definition \ref{def:HDM:GR}, with $B$ satisfying Estimate \eqref{coer:Bs} and $(V_h,I_h,Q_h,\stab_h)$ satisfying \assum{0}--\assum{5}.
Then, there exists a constant $C$, not depending on $h$, such that
	\begin{itemize}
		\item  $C_\disc^B \le C$,
		\item $ \forall\: \varphi \in W$, $S_{\disc}^B(\varphi) \le Ch\norm{\varphi}{W}$,
		\item $\forall \:\xi \in H^2(\O)^{d\times d}$, $W_{\disc}^B(\xi) \le Ch \norm{\xi}{H^2(\O)^{d\times d}}.$
	\end{itemize}
\end{theorem}

Before proving this theorem, let us note the following straightforward consequence of Remark \ref{rem:dens}.

\begin{corollary}[Properties of Hessian discretisation based on gradient recovery]\label{cor:HD.coerciveB}~\\
	Let $(\disc_m)_{m \in \N}$ be a sequence of $B\textendash$Hessian discretisations, with $B$ satisfying Estimate \eqref{coer:Bs} and each $\disc_m$ associated with $(V_{h_m},Q_{h_m},I_{h_m},\stab_{h_m})$
satisfying \assum{0}--\assum{5} uniformly with respect to $m$. Assume that $h_m\to 0$ as $m\to\infty$.
Then the sequence $(\disc_m)_{m\in\N}$ is coercive, consistent and limit-conforming.
\end{corollary}

\medskip

\begin{proof}[Proof of Theorem \ref{thm:HD.coerciveB}]~\smallskip

$\bullet$ \textsc{Coercivity}: Let $v\in X_{\disc,0}$. Noticing that $|a\otimes b|=|a| |b|$ for any two vectors $a$ and $b$, the definition of $\hbd$, Property \eqref{coer:Bs} of $B$ and $|\stab|\ge 1$
yield
\begin{align}
\norm{\hbd v}{}^2\ge{}&C_B^2\int_\O \left|\nabla(Q_h\nabla v)+\stab_h\otimes (Q_h\nabla v-\nabla v)\right|^2\d\x\nonumber\\
={}&C_B^2\int_\O |\nabla (Q_h \nabla v)|^2 \d\x +C_B^2\int_{\O}^{}|\stab_h\otimes (Q_h\nabla v-\nabla v)|^2\d\x\nonumber \\ 
&+2C_B^2\int_{\O}^{}\nabla (Q_h \nabla v):\stab_h\otimes (Q_h\nabla v-\nabla v)\d\x \nonumber\\
\ge{}& C_B^2\left(\norm{\nabla (Q_h \nabla v)}{}^2+\norm{Q_h\nabla v-\nabla v}{}^2\right)\nonumber \\
& +2C_B^2\sum_{\cell \in \mesh}^{}\int_{\cell}\nabla (Q_h \nabla v):\stab_h\otimes (Q_h\nabla v-\nabla v)\d\x.\nonumber
\end{align}
Since $\nabla (Q_h \nabla v)_{|\cell} \in \nabla V_h(\cell)^d $, a use of property $\assum{5}$ 
shows that the last term vanishes, and we have thus
\be\label{seminorm_qh0}
\norm{\hbd v}{}^2\ge C_B^2 \left(\norm{\nabla (Q_h \nabla v)}{}^2+\norm{Q_h\nabla v-\nabla v}{}^2\right),
\ee
which implies
\be \label{seminorm_qh}
C_B^{-1}\sqrt{2}\norm{\hbd v}{}\ge  \norm{\nabla (Q_h \nabla v)}{}+\norm{Q_h\nabla v-\nabla v}{}.
\ee
Apply now the Poincar\'{e} inequality twice, the triangle inequality and \eqref{seminorm_qh} to obtain
\begin{align}
\norm{\Pi_\disc v}{}=\norm{v}{}\le{}& {\rm diam}(\O)\norm{\nabla v}{}\nonumber\\
\le{}& {\rm diam}(\O)\norm{\nabla v-Q_h\nabla v}{}+{\rm diam}(\O)\norm{Q_h\nabla v}{}\nonumber\\
\le{}& {\rm diam}(\O)\norm{\nabla v-Q_h\nabla v}{}+{\rm diam}(\O)^2\norm{\nabla (Q_h\nabla v)}{}\nonumber\\
\le{}& C_B^{-1}\sqrt{2}\max({\rm diam}(\O),{\rm diam}(\O)^2)
\norm{\hbd v}{}.\label{coercive_gradrec1}
\end{align}
From \eqref{seminorm_qh0} and the Poincar\'{e} inequality, we also have
\be \label{coercive_gradrec2}
\norm{\nabla_\disc v}{}=\norm{Q_h\nabla v}{}
\le {\rm diam}(\O)\norm{\nabla(Q_h\nabla v)}{}
\le {\rm diam}(\O)C_B^{-1}\norm{\hbd v}{}.
\ee
Estimates \eqref{coercive_gradrec1} and \eqref{coercive_gradrec2} show that $C_\disc^B\le 
C_B^{-1}\sqrt{2}\max({\rm diam}(\O),{\rm diam}(\O)^2)$.

\smallskip

$\bullet$ \textsc{Consistency}: let $\varphi \in W\subset H^3(\O)\cap H^2_0(\O)$ and  choose $v=I_h\varphi \in X_{\disc,0}$. 
Using the properties \assum{0} (which implies $\norm{I_h\varphi-\varphi}{}\le Ch\norm{\varphi}{H^2(\O)}$ by the Poincar\'e inequality) and \assum{2}, we obtain
\be \label{consistency_gradrec1}
\norm{\Pi_\disc v-\varphi}{}=\norm{I_h\varphi-\varphi}{}  \le Ch\norm{\varphi}{{H^2}(\O)}
\ee
and
\be \label{consistency_gradrec2}
\norm{\nabla_\disc v-\nabla \varphi}{}=\norm{Q_h\nabla I_h\varphi-\nabla \varphi}{}\le Ch^2\norm{\varphi}{W}.
\ee
Let us now turn to $\norm{\hbd v-\hb \varphi}{}$. Observe that $\nabla\nabla$ is another notation for $\hessian$. Using a triangle inequality, the boundedness of $B$ and $\stab_h$
implies
\begin{align}
\norm{\hbd v-\hb \varphi}{}={}&\norm{B\left[\nabla (Q_h \nabla v)+\stab_h\otimes (Q_h\nabla v-\nabla v)\right]-B\hessian \varphi}{}\nonumber \\
\le{}& \norm{B\left[\nabla (Q_h \nabla v)-\nabla\nabla \varphi\right]}{} + \norm{B\stab_h\otimes (Q_h\nabla v-\nabla v)}{} \nonumber \\\
\le{}& C \underbrace{\norm{\nabla (Q_h \nabla v)-\nabla\nabla \varphi}{}}_{A_1} +C\underbrace{\norm{Q_h\nabla v-\nabla v}{}}_{A_2}. \label{consistency_gradrec}
\end{align}
Introducing the term $\nabla (Q_h \nabla \varphi)$, using in sequence the triangle inequality, the inverse inequality in \assum{0}, \assum{3}, the projection property of $Q_h$, \assum{1} and \assum{2}, we get
\begin{align}
A_1&\le \norm{\nabla [Q_h \nabla v-Q_h \nabla \varphi]}{}+\norm{\nabla (Q_h \nabla \varphi)-\nabla\nabla \varphi}{}\nonumber\\
&\le Ch^{-1}\norm{ Q_h \nabla v-Q_h \nabla \varphi}{}+Ch\norm{\nabla \varphi}{H^2(\O)} \nonumber \\
&\le Ch^{-1}\norm{Q_h \left( Q_h \nabla v-\nabla \varphi\right)}{}+Ch\norm{\nabla \varphi}{H^2(\O)} \nonumber \\
&\le  Ch^{-1}\norm{ Q_h \nabla I_h\varphi-\nabla \varphi}{}+Ch\norm{\nabla \varphi}{H^2(\O)} \le Ch 
\norm{\varphi}{W}.\label{consistency_gradrec:a1}
\end{align}
To estimate $A_2$, we use the properties  \assum{2} and \assum{0}:
\be
A_2 \le \norm{Q_h\nabla v-\nabla \varphi}{}+\norm{\nabla \varphi-\nabla v}{} 
\le Ch^2\norm{\varphi}{W}+Ch\norm{\varphi}{H^2(\O)}.\label{consistency_gradrec:a2}
\ee
The estimate on $S^B_\disc(\varphi)$ follows from \eqref{consistency_gradrec1}--\eqref{consistency_gradrec:a2}.

\smallskip

$\bullet$ \textsc{Limit-conformity}: for $ \xi \in H^2(\O)^{d\times d}$ and $v \in X_{\disc,0},$
\begin{align}
\int_\O \Big((\hessian:B^{\tau}B\xi)\Pi_\disc v {}&- B\xi:\hbd v\Big)\d\x\nonumber\\
={}&\underbrace{\int_\O \Big((\hessian:B^{\tau}B\xi)\Pi_\disc v - B\xi:B\nabla (Q_h \nabla v)\Big)\d\x}_{B_1}\nonumber
\\
&-\underbrace{\int_\O B \xi:B\stab_h\otimes (Q_h\nabla v-\nabla v)\d\x}_{B_2}. \label{limitconformity_gradrec12}
\end{align}
Recall that $v=\Pi_\disc v$ and $A=B^\tau B$. Since $Q_h\nabla v\in H^1_0(\O)$, Lemma \ref{IBP} applied to $(\hessian:A\xi)v$ and an integration-by-parts on $B\xi:B\nabla (Q_h \nabla v)=A\xi:\nabla(Q_h\nabla v)$ show that, for any $\mu_h\in N_h=[(Q_h\nabla-\nabla)(V_h)]^\perp$,
\begin{align}
|B_1|&=\Big\lvert\int_\O(\hessian:A\xi) v\d\x +\int_\O Q_h \nabla v\cdot\mbox{div}(A\xi)\d\x\Big\rvert \nonumber\\
&=\Big\lvert\int_\O(Q_h \nabla v-\nabla v)\cdot\mbox{div}(A\xi)\d\x\Big\rvert\nonumber\\
&=\Big\lvert\int_\O(Q_h \nabla v-\nabla v)\cdot(\mbox{div}(A\xi)-\mu_h)\d\x\Big\rvert\nonumber\\
&\le\norm{Q_h \nabla v-\nabla v}{}\norm{\mbox{div}(A\xi)-\mu_h}{}.
\label{limitconf_gradrec_inf.b1}
 \end{align}
Take the infimum over all $\mu_h \in N_h$. Estimate \eqref{seminorm_qh} and Property \assum{4} yield
\be \label{limitconformity_gradrec1}
|B_1|\le Ch\norm{\hbd v}{}\norm{\mbox{div}(A\xi)}{H^1(\O)^d}.
 \ee
Let $\xi_\cell$ denote the average of $\xi$ over $\cell \in \mesh$. By the mesh regularity assumption, $\norm{\xi-\xi_\cell}{L^2(\cell)^{d\times d}}\le Ch\norm{\xi}{H^1(\cell)^{d\times d}}$ (see, e.g., \cite[Lemma B.6]{koala}).
Moreover, since $V_h$ contains the piecewise constant functions, $\nabla V_h(K)$ contains
the constant vector-valued functions on $K$ and thus, by the orthogonality condition in \assum{5}, the Cauchy-Schwarz inequality, the boundedness of $B$ and $\stab_h$, and \eqref{seminorm_qh},
\begin{align}
|B_2|={}&\Big\lvert\sum_{\cell \in \mesh}^{}\int_\cell B^\tau B\xi:\stab_h\otimes (Q_h\nabla v-\nabla v)\d\x\Big\rvert\nonumber\\
={}&\Big\lvert\sum_{\cell \in \mesh}^{}\int_\cell(B^\tau B\xi-B^\tau B\xi_\cell):\stab_h\otimes (Q_h\nabla v-\nabla v)\d\x\Big\rvert\nonumber\\
\le{}& C \sum_{\cell \in \mesh}^{}\norm{\xi-\xi_\cell}{L^2(\cell)}\norm{Q_h\nabla v-\nabla v}{L^2(\cell)}\nonumber\\
\le{}& Ch\norm{\xi}{H^1(\O)^{d\times d}}\norm{\hbd v}{}. \label{limitconformity_gradrec2}
\end{align}
Plugging \eqref{limitconformity_gradrec1} and \eqref{limitconformity_gradrec2} into \eqref{limitconformity_gradrec12} yields
\begin{multline*}
\Big\lvert\int_\O \Big((\hessian:B^{\tau}B\xi)\Pi_\disc v - B\xi:\hbd v\Big)\d\x \Big\rvert\\
\le Ch\Big(\norm{\mbox{div}(A\xi)}{H^1(\O)^d}+\norm{\xi}{H^1(\O)^{d\times d}}\Big)\norm{\hbd v}{}.
\end{multline*}
By the definition \eqref{def.WD} of $W_\disc^B(\xi)$, this concludes the proof of the estimate
on this quantity.
\end{proof}

\subsection{A gradient recovery operator based on biorthogonal systems}\label{sec:biorth}

We present here a particular case of a method based on a gradient recovery operator, using biorthogonal systems as in \cite{BL_FEM}. $V_h$ is the conforming $\P1$ FE space on a mesh of simplices, and $I_h$ is the Lagrange interpolation with respect to vertices of $\mesh$.
We will build a locally computable projector $Q_h$, that is, such that determining $Q_h f$ on a cell $K$ only requires the knowledge of $f$ on $K$ and its neighbouring cells.

Let ${\mathcal B}_1:=\{\phi_1,\cdots,\phi_n\}$ be the set of 
basis functions of $V_h$ associated with the inner vertices in $\mesh$. Let 
the set ${\mathcal B}_2:=\{\psi_1,\cdots,\psi_n\}$  be the set of discontinuous piecewise linear functions 
biorthogonal to ${\mathcal B}_1$ also associated with the inner vertices of $\mesh$,
so that elements of ${\mathcal B}_1$ and 
${\mathcal B}_2$ satisfy the biorthogonality relation
\be \label{biorth}
  \int_{\Omega} \psi_i  \phi_j \d\x = c_j \delta_{ij},
\; c_j\neq 0,\; 1\le i,j \le n,
\ee
 where $\delta_{ij}$ is 
 the Kronecker symbol and $c_j = \int_{\Omega}\psi_j \phi_j \,d\x.$
 Let  $M_h:=\text{span}\{{\mathcal B}_2\}$. 
 Such biorthogonal systems have been constructed in the context of 
 mortar finite elements, and later extended to gradient recovery operators
\cite{KLP01,Lam11b,BL_FEM}.
The basis functions of $M_h$ can be defined on a reference element. 
For example, for the reference triangle, we have 
\[
  \hat \psi_1(\x):=3-4x_1-4x_2,\quad
  \hat\psi_2(\x):=4x_1-1,\quad\text{and}\quad
  \hat\psi_3(\x):=4x_2-1,
\]
associated with its three vertices $(0,0)$, $(1,0)$ and $(0,1)$, 
respectively. For the refe\-rence tetrahedron, we have 
\begin{align*}
  &\hat \psi_1(\x):=4-5x_1-5x_2-5x_3,\quad
  \hat\psi_2(\x):=5x_1-1,\\
  &\hat\psi_3(\x):=5x_2-1,\quad\text{and}\quad
\hat\psi_4(\x):=5x_3-1,
\end{align*}
associated with its four vertices 
$(0,0,0)$, $(1,0,0)$, $(0,1,0)$ and $(0,0,1)$, respectively.  
These basis functions satisfy 
\begin{equation}\label{constantrep} 
\sum_{i=1}^{d+1} \hat \psi_i=1.
\end{equation}

The projection operator $Q_h : L^2(\Omega) 
\rightarrow V_h$  is the oblique projector onto $V_h$ defined as: for $f\in L^2(\O)$,
$Q_hf\in V_h$ satisfies
\begin{equation}\label{qdef}
\int_{\Omega} (Q_hf)\, \psi_h\d\x = \int_{\Omega} f\,\psi_h\d\x,\quad \forall\psi_h \in M_h.
\end{equation}
Due to the biorthogonality relation \eqref{biorth}, $Q_h$ is well-defined and has the explicit representation
\begin{equation}\label{eq2}
 Q_hf = \sum_{i=1}^n\frac{\int_{\Omega} \psi_i \,f\d\x}{c_i} \phi_i.
\end{equation}

The relation \eqref{qdef} shows $M_h\subset [(Q_h-I)(L^2(\O))]^\bot$. Hence, if $M_h$ satisfies the approximation property 
\[ \inf_{\alpha_h\in  M_h}\norm{\alpha_h-\psi}{}\le Ch\norm{\psi}{H^1(\O)}, \quad  \forall 
\psi\in H^1(\O),\] we know that \assum{4} holds.  
In order to get this approximation property it is sufficient that the basis functions of $M_h$ reproduce constant functions. 
Let $K \in \mesh$ be an interior element not touching any boundary vertex. 
Due to the property \eqref{constantrep} 
\[  \sum_{i=1}^{d+1} \psi_{\vertex_i} = 1 
\quad\text{on}\quad K,
\] 
where  $\{\psi_{\vertex_i} \}_{i=1}^{d+1}$ are  
basis functions of  $M_h$ associated 
with the vertices $(\vertex_1,\ldots,\vertex_{d+1})$ of $K$.

However, this property 
does not hold on $K\in \mesh$ if $K$ has one or more vertices on the boundary. 
We need to modify  the piecewise linear basis functions of 
$M_h$  to guarantee the approximation property
\cite{LSW05,Lam06}.  Let  $W_h \subset H^1(\Omega)$ be 
the lowest order FE space including the basis functions on the boundary vertices of $\mesh$, and let 
$\widetilde M_h$ the space spanned by the discontinuous basis functions 
biorthogonal to the basis functions of $W_h$. 
$M_h$ is then obtained as a modification of $\widetilde M_h$, by moving
all vertex basis functions of this latter space to 
 nearby internal vertices using the following three steps. 
\begin{enumerate}
\item For a basis function $\widetilde \psi_k$  of $\widetilde M_h$ associated with a vertex $\vertex_k$ 
on the boundary we find a closest \emph{internal} triangle or tetrahedron $K \in \mesh$
(that is, $K$ does not have a boundary vertex).
\item Compute the barycentric coordinates $\{\alpha_{K,i} \}_{i=1}^{d+1}$ 
 of $\vertex_k$ with respect to the vertices of $K$, and modify all the basis functions
$\{\widetilde \psi_{K,i}\}_{i=1}^{d+1}$ of $\widetilde M_h$ associated
with $K$ into $  \psi_{K,i}= \widetilde \psi_{K,i} + \alpha_{K,i} \widetilde \psi_k$ for $i=1,\cdots, d+1$.
\item Remove $\widetilde \psi_k$ from the basis of $\widetilde M_h$.
\end{enumerate}

An alternative way is to modify the basis functions of all triangles or tetrahedra having one or more 
boundary vertices  as proposed in \cite{KLP01}.
 \begin{enumerate}
 \item If all vertices $\{\vertex_i\}_{i=1}^{d+1}$ of an element  $K\in \mesh$ are inner vertices, then the 
linear basis functions 
$\{\psi_{\vertex_i}\}_{i=1}^{d+1}$ of  $M_h$  on $K$ 
are defined using  the biorthogonal relationship \eqref{biorth} 
with the basis functions $\{\phi_{\vertex_i}\}_{i=1}^{d+1}$ of  $V_h$. 
\item If an element $K \in \mesh$ has all boundary vertices, then 
we find a neighbouring element $\widetilde K$, which has at least one inner
vertex $\vertex$, and we extend the support of the basis function $\psi_{\vertex} \in M_h$
associated with $\vertex$ to the element $K$ by defining 
$\psi_{\vertex} =1$ on $K$.
\item If an element $K \in \mesh$ has only one inner vertex $\vertex$
 and other boundary vertices,   then the basis function  $\psi_{\vertex} \in M_h$ associated with the inner vertex $\vertex$ 
is defined as $\psi_{\vertex} =1$ on $K$.
\item If an element $K$ has two inner vertices $\vertex_1$ and $\vertex_2$ and other boundary vertices, then the basis functions 
$\psi_{\vertex_1}, \, \psi_{\vertex_2} \in M_h$ associated with these points are chosen to satisfy the 
biorthogonal  relationship \eqref{biorth} with $\phi_{\vertex_1}, \, \phi_{\vertex_2} \in V_h$,
as well as the property $\psi_{\vertex_1} + \psi_{\vertex_2} = 1$ on $K$. 

\item In the three-dimensional case, we can have 
an element $K$ with three inner vertices $\{\vertex_i\}_{i=1}^3$ and one boundary vertex.  In this case we define three basis functions 
$\{\psi_{\vertex_i}\}_{i=1}^3$ to satisfy the biorthogonal relationship \eqref{biorth} with $\{\phi_{\vertex_i}\}_{i=1}^3$
 as well as the condition $\sum_{i=1}^3 \psi_{\vertex_i} = 1$ on $K$.
\end{enumerate}
The projection $Q_h$ is stable in $L^2$ and $H^1$-norms \cite{Lam11b}, and hence 
assumption \assum{1} follows. To establish \assum{2}, we need the following mesh assumption.
 \begin{itemize}
	\item[\textbf{(M)}]For any vertex $\vertex$, denoting by $\mesh_{\vertex}$ the set of cells having $\vertex$ as
a vertex,
  \[
	\sum\limits_{K \in \mesh_\vertex} \frac{|K|}{|S_\vertex|}(\overline{\x}_K-\vertex) = O(h^{2}),
	\] 
where $S_\vertex$ is the support of the 
basis function $\phi_\vertex$ of $V_h$ associated with $\vertex$.
\end{itemize}
This assumption is satisfied if the triangles of the mesh can be paired
in sets of two that share a common edge and form an $O(h^2)$-parallelogram, that is,
the lengths of any two opposite edges differ only by $O(h^2)$.
In three dimensions, \textbf{(M)} is satisfied if the lengths of each pair of opposite edges of a given element are allowed to differ only by $O(h^2)$ \cite{CWD14}.
The following theorem establishes \assum{2} with $W=W^{3,\infty}(\Omega)\cap H^2_0(\O)$ and can be proved as in \cite{XZ04,Lam11b}.

\begin{theorem} \label{thnew2}
Let $ u \in W^{3,\infty}(\Omega)\cap H^2_0(\O)$. 
Assume that the triangulation satisfies the assumption {\rm\textbf{(M)}}.
Then \[ 
\norm{ Q_h \nabla I_h u - \nabla u}{} \leq 
C h^2 \|u\|_{W^{3,\infty}(\Omega)}
.\]
\end{theorem}
Since $Q_h$ is a projection onto $V_h$, $Q_hI_h=I_h$. Hence, for $w\in H^2(\O)\cap H^1_0(\O)$,
introducing $Q_h I_h w=I_hw$ and invoking the $H^1$-stability property of $Q_h$ \cite[Lemma 1.8]{Lam06} leads to
\[
\norm{\nabla Q_h w-\nabla w}{}\le \norm{\nabla Q_h (w-I_hw)}{} 
+\norm{\nabla I_h w-\nabla w}{}\le C \norm{\nabla I_h w-\nabla w}{}.
\]
The standard approximation properties of $V_h$ then guarantee \assum{3}.
The Assumption \assum{4} is satisfied since $M_h\subset N_h$ (${M}_h$ is obtained by combining functions in $\widetilde M_h$, that satisfies this property) and the basis functions of $M_h$ locally reproduce constant functions.
To build $\stab_h$ that satisfies \assum{5}, divide each triangle $K \in \mesh$ into four equal triangles using the mid-points of each side, and define $\stab_h$ as a piecewise constant function as described in Figure \ref{fig:stab}. It can be checked that this function satisfies \assum{5}. A similar construction also works on tetrahedra (in which case $\stab_{h|K}$ is equal to $1$ on the four sub-tetrahedra constructed around the vertices of $K$, and $-4$ in the rest of $K$).

\begin{figure}[h!]
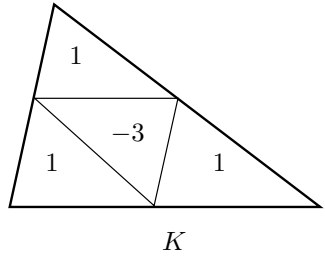
\caption{Values of the stabilisation function $\stab_h$ inside a cell $K$.}
\label{fig:stab}
\end{figure}

\section{Finite volume method based on $\Delta$-adapted discretizations}\label{sec.fvm}

We consider here the finite volume (FV) scheme from \cite{biharmonicFV} for the biharmonic problem \eqref{biharmonic} on $\Delta$-adapted meshes, that is, meshes that satisfy an orthogonality property.

\begin{definition}[$\Delta$-adapted FV mesh]\label{FV.def}
	A general mesh $\polyd$ is $\Delta$-adapted if
	\begin{enumerate}
		\item for all  $\edge \in \edgesint$, denoting by $\cell, L \in \mesh$ the cells such that $\mesh_{\edge}= \{\cell, L\}$, the straight line $(\x_\cell, \x_L)$ intersects and is orthogonal to $\edge$,
		\item for all $\edge \in \edgesext$ with $\mesh_{\edge}= \{\cell\}$, the line orthogonal to $\edge$ going through $\x_\cell$ intersects $\edge$.
	\end{enumerate}
\end{definition}
For such a mesh, we let $D_{\cell,\edge}$ be the cone with vertex $\x_\cell$ and basis $\edge$, and $D_{\edge}=\bigcup_{\cell \in \mesh_{\edge}}D_{\cell,\edge}$.
For each $\edge \in \edgesint$, an orientation is chosen by defining one of the two unit normal vectors $n_{\edge}$, and we denote by $\cell^-_{\edge}$ and $\cell^+_{\edge}$ the two adjacent control volumes such that $n_{\edge}$ is oriented from $\cell^-_{\edge}$ to $\cell^+_{\edge}$.
For all $\edge \in \edgesext$, we denote the control volume $\cell \in \mesh$ such that $\edge \in \edgescv$ by $\cell_{\edge}$ and we define $n_{\edge}$ by $n_{\cell,\edge}$. We then set
\begin{equation}
d_\edge=\left\{
\begin{array}{ll}
\dist(\x_{\cell_\edge^-},\edge) + \dist(\x_{\cell_\edge^+},\edge) & \forall \edge \in \edgesint \\
\dist(\x_{\cell},\edge)& \forall \edge \in \edgesext.
\end{array}\right.
\label{dsigma}
\end{equation} 
For all $\cell \in \mesh$, set $\edges_{\cell,\rm{int}} = \edgescv \cap \edgesint$ and $\edges_{\cell,\rm{ext}} = \edgescv \cap \edgesext$. Finally, we define the mesh regularity factor by
$$\theta_\polyd = \max\Bigg\{\max\left(\frac{\mbox{diam}(\cell)}{\dist(\x_\cell,\edge)}, \frac{d_\edge}{\dist(\x_\cell,\edge)}\right)\,;\, \cell \in \mesh, \edge \in \edgescv \Bigg\}.$$

We now define a notion of $B$--Hessian discretisation for $B=\frac{{\rm tr}(\cdot)}{\sqrt{d}}{\rm Id}$, in which case \eqref{weak} corresponds to the biharmonic problem \eqref{biharmonic}, for which the coercivity property \eqref{coer:B} holds (see Section \ref{sec:bihar}).

\begin{definition}[$B\textendash$Hessian discretisation based on $\Delta$-adapted discretisation]\label{def.HD.Deltaadapted}
	Let $B=\frac{{\rm tr}(\cdot)}{\sqrt{d}}{\rm Id}$ and $\polyd$ be a $\Delta$-adapted mesh. A $B\textendash$Hessian discretisation is given by $\disc=(X_{\disc,0},\Pi_\disc,\nabla_\disc,\hbd)$ where
	\begin{itemize}
		\item $X_{\disc,0}$ is the space of all real families $u_\disc=(u_\cell)_{\cell \in \mesh}$, such that $u_\cell=0$ for all $\cell \in \mesh$ with $\edges_{\cell,\rm ext} \ne \emptyset$.
		\item For $u_\disc \in X_{\disc,0}$, $\Pi_\disc u_\disc$ is the piecewise constant function equal to $u_\cell$ on the cell $\cell$. 
		\item The discrete gradient $\nabla_\disc u_\disc$ is defined by its constant values on the cells:
		\begin{equation}
		\nabla_\cell u_\disc=\frac{1}{|\cell|}\sum_{\edge \in \edges_\cell}^{}\frac{|\edge|(\delta_{\cell,\edge} u_\disc)(\centeredge-\x_\cell)}{d_\edge}, \label{biharmonic_nabla}
		\end{equation}
		where \begin{equation}
		\delta_{\cell,\edge} u_\disc=\left\{\begin{array}{ll}
		u_L-u_\cell &\forall \,\edge \in \edges_{\cell,\rm{int}}\,,\; \mesh_{\edge}= \{\cell, L\} \\
		0 & \forall \, \edge \in \edges_{\cell,\rm{ext}}.
		\end{array}\right.
		\label{jump}
		\end{equation} 
		\item The discrete Laplace operator $\Delta_\disc$ is defined by its constant values on the cells:
		\begin{equation}
		\Delta_\cell u_\disc=\frac{1}{|\cell|}\sum_{\edge \in \edges_\cell}^{}\frac{|\edge|\delta_{\cell,\edge} u_\disc}{d_\edge}. \label{biharmonic_delta}
		\end{equation}
		We then set $\hbd u_\disc = \frac{\Delta_\disc u_\disc}{\sqrt{d}}{\rm Id}$.
	\end{itemize}
\end{definition}
For $u_\disc, v_\disc \in X_{\disc,0},$
\begin{equation}
[u_\disc,v_\disc]=\sum_{\edge \in \edges}^{}\frac{|\edge|\delta_{\edge}u_\disc\delta_{\edge} v_\disc}{d_\edge} \label{innerproduct}
\end{equation}
defines an inner product on $X_{\disc,0}$, whose associated norm is denoted by $\norm{u_\disc}{\disc}$. Here $\delta_{\edge}$ is given by
\begin{equation}
\delta_{\edge} u_\disc=
\left\{\begin{array}{ll}
u_{\cell^+_{\edge}}-u_{\cell^-_{\edge}}& \forall \,\edge \in \edgesint \\
0 & \forall \, \edge \in \edgesext.
\end{array}\right.\label{jump1}
\end{equation} 
It can easily be checked that, with this Hessian discretisation, the Hessian scheme \eqref{weak}
is the scheme of \cite{biharmonicFV} for the biharmonic equation.
Let us examine the properties of this Hessian discretisation.
\begin{theorem}\label{thm:HD.Deltaadapted}
	Let $\disc$ be a $B\textendash$Hessian discretisation in the sense of Definition \ref{def.HD.Deltaadapted}. Then there exists a constant $C$, depending only on on $\theta\ge \theta_\polyd$, such that
	\begin{itemize}
		\item  $C_\disc^B \le C$,
		\item  If $\varphi \in C^2_c(\Omega)$, $\Delta \varphi\in H^1(\Omega)$ and $a>0$ is such that ${\rm supp}(\varphi)\subset \{x\in\Omega\,;\,\dist(\x,\partial\Omega)>a\}$, then
			\begin{equation}\label{est.SD.Cc}
			S_{\disc}^B(\varphi) \le Ch\norm{\Delta\varphi}{H^1(\Omega)}+Ch \norm{\varphi}{C^2(\overline{\Omega})}
			\times\left\{\begin{array}{ll}|\ln(a)|a^{-3/2}&\mbox{ if $d=2$},\\
			a^{-5/3}&\mbox{ if $d=3$}.\end{array}\right.
			\end{equation}
		
		\item If $\varphi\in H^2_0(\Omega)\cap C^2(\overline{\Omega})$ with $\Delta\varphi\in H^1(\Omega)$, then
			\begin{equation}\label{est.SD.FV}
			S_{\disc}^B(\varphi) \le  Ch\norm{\Delta\varphi}{H^1(\Omega)} +
			C \norm{\varphi}{C^2(\overline{\Omega})}\times\left\{\begin{array}{ll}h^{1/4}|\ln(h)|&\mbox{ if $d=2$},\\
			h^{3/13}&\mbox{ if $d=3$}.\end{array}\right.
			\end{equation}
		
		\item $\forall\xi \in H^2(\O)^{d \times d}$, $W_{\disc}^B(\xi) \le Ch \norm{{\rm tr}(\xi)}{H^2(\O)}$.
	\end{itemize}
\end{theorem}

\begin{remark}\label{FVM.remark} If the solution $\bu$ to \eqref{biharmonic} belongs to $H^4(\Omega)\cap H^2_0(\Omega)$, then $\bu\in C^2(\overline{\Omega})$ and $\Delta \bu\in H^2(\Omega)$. In that case, Theorems \ref{error} and \ref{thm:HD.Deltaadapted} provide an $O(h^{1/4}|\ln(h)|)$ (in dimension $d=2$) or $\mathcal O(h^{3/13})$ (in dimension $d=3$) error estimate for the Hessian scheme based on the HD from Definition \ref{def.HD.Deltaadapted}.
		This slightly improves the result of \cite[Theorem 4.3]{biharmonicFV}, in which an $O(h^{1/5})$ estimate is obtained if $\bu\in C^4(\overline{\O})\cap H^2_0(\O)$.
\end{remark}
As for the method based on gradient recovery operators, the properties of the Hessian discretisation follow from
the estimates in Theorem \ref{thm:HD.Deltaadapted} and from Remark \ref{rem:dens}.
\begin{corollary}\label{cor:HD.Deltaadapted}
	Let $(\disc_m)_{m \in \N}$ be a sequence of $B\textendash$Hessian discretisations in the sense of Definition \ref{def.HD.Deltaadapted}, associated to meshes
	such that $h_{m}\to 0$ and $(\theta_{\polyd_m})_{m\in\N}$ is bounded.
	Then the sequence $(\disc_m)_{m \in \N}$ is coercive, consistent and limit-conforming.	
\end{corollary}

\begin{proof}[Proof of Theorem \ref{thm:HD.Deltaadapted}]~\\
	
	$\bullet$ \textsc{Coercivity}: the discrete Poincar\'e inequality of \cite{EGH00} states that
	\begin{equation}
	\norm{\Pi_\disc v_\disc}{} \le \mbox{diam}(\O)\norm{v_\disc}{\disc}\,, \quad \forall  v_\disc \in X_{\disc,0}. \label{PI}
	\end{equation}
	Let us first prove that
	\be
	-\int_{\O}^{}\Pi_\disc u_\disc \Delta_\disc v_\disc dx = [u_\disc, v_\disc]_\disc,\quad u_\disc,\,  v_\disc \in X_{\disc,0}.\label{int_ip}
	\ee
	The definitions of $\Pi_\disc$ and $\Delta_\disc$ yield
	\[
	-\int_{\O}^{}\Pi_\disc u_\disc \Delta_\disc v_\disc dx 
	=\sum_{\cell \in \mesh}^{}-|\cell|u_\cell \Delta_\cell v_\disc=-\sum_{\cell \in \mesh}^{}u_\cell\sum_{\edge \in \edges_\cell}^{}\frac{|\edge|\delta_{\cell,\edge} v_\disc}{d_\edge}. 
	\]
	For $\edge \in \edgesext$, $\delta_{\cell,\edge}v_\disc=0.$ Gathering the sums by edges and using \eqref{jump} and \eqref{jump1}, we obtain
	\[
	-\int_{\O}^{}\Pi_\disc u_\disc \Delta_\disc v_\disc dx =\sum_{\cell \in \mesh}^{}u_\cell\sum_{\edge \in \edges_{\cell,\rm{int}}}^{}\frac{|\edge|(v_\cell-v_L)}{d_\edge} 
	=\sum_{\edge \in \edgesint}^{}\frac{|\edge|\delta_{\edge}u_\disc\delta_{\edge} v_\disc}{d_\edge},
	\]
	which establishes \eqref{int_ip}.
	Choosing $v_\disc=u_\disc$, applying the Cauchy--Schwarz inequality and using \eqref{PI}, we get
	\begin{equation*}
	\norm{u_\disc}{\disc}^2 \le \norm{\Pi_\disc u_\disc}{} \norm{\Delta_\disc u_\disc}{} 
	\le \mbox{diam}(\O)\norm{u_\disc}{\disc} \norm{\Delta_\disc u_\disc}{}.
	\end{equation*}
	Thus,
	\be \norm{u_\disc}{\disc} \le \mbox{diam}(\O)\norm{\Delta_\disc u_\disc}{}. \label{ineq_1}\ee
	Combining \eqref{PI} and \eqref{ineq_1}, we get
	\begin{equation}
	\norm{\Pi_\disc v_\disc}{} \le \mbox{diam}(\O)^2 \norm{\Delta_\disc u_\disc}{}. \label{ineq_2}
	\end{equation}
	The stability of the discrete gradient \cite[Lemma 4.1]{biharmonicFV} yields
	\[
	\norm{\nabla_\disc u_\disc}{} \le \theta\sqrt{d}\norm{u_\disc}{\disc}\quad\forall u_\disc \in X_{\disc,0}.
	\]
	Estimate \eqref{ineq_1} then shows that $\norm{\nabla_\disc u_\disc}{} \le \mbox{diam}(\O)\theta \sqrt{d}\norm{\Delta_\disc u_\disc}{}$, which, together with \eqref{ineq_2}, concludes the proof of the estimate on $C_\disc^B$.
	\smallskip
	
	$\bullet$ \textsc{Consistency -- compact support}:
	The proof utilises the ideas of \cite{biharmonicFV}, with a few improvements of the estimates. For $s>0$ we let $\Omega_s=\{x\in\Omega\,;\,\dist(\x,\partial\Omega)>s\}$. In this proof, $A\lesssim B$ means that $A\le CB$ for some constant $C$ depending only on $\theta$.
		
		We first consider the case where $\varphi\in C^2_c(\Omega)$ and $\Delta\varphi\in H^1(\Omega)$, with  support at distance from $\partial\Omega$ equal to or greater than $a$. As in \cite[Proof of Lemma 4.4]{biharmonicFV}, let $\psi^a\in C^\infty_c(\Omega)$, equal to $1$ on $\Omega_{3a/4}$, that vanishes on $\Omega\backslash\Omega_{a/4}$, and such that, for all $\alpha\in\N^d$, with $|\alpha|=\sum_{i=1}^d\alpha_i$,
		\begin{equation}\label{prop.psia}
		\norm{\partial^\alpha \psi^a}{L^\infty(\Omega)}\lesssim a^{-|\alpha|}.
		\end{equation}
		Letting $\psi^a_\disc=(\psi^a(\x_K))_{K\in\mesh}$, we have
		$|\Delta_\disc \psi^a_\disc|\lesssim a^{-2}$. Hence, for all $r\in [1,\infty]$, since $\Omega\backslash\Omega_{2a}$ has measure $\lesssim a$,
		\begin{equation}\label{est.psia}
		\norm{\Delta_\disc \psi^a_\disc}{L^r(\Omega)}\lesssim a^{-2+\frac{1}{r}}.
		\end{equation}
		
		Letting $\widetilde{v}=(\widetilde{v}_K)_{K\in\mesh}$ be the solution of the two-point flux approximation finite volume scheme with homogeneous Dirichlet boundary conditions and source term $-\Delta \varphi$, by \cite{EGH00} we have, with $\varphi_\disc=(\varphi(\x_K))_{K\in\mesh}$,
		\begin{equation}
		\label{est.H1}
		\left(\sum_{\sigma \in \edges}\frac{|\sigma|}{d_\sigma} (\delta_\sigma(\widetilde{v}-\varphi_\disc))^2\right)^{1/2}\lesssim h \norm{\varphi}{C^2(\overline{\Omega})}
		\end{equation}
		and, for $q\in[1,+\infty)$ if $d=2$, $q\in[1,6]$ if $d=3$,
		\begin{equation}
		\label{est.Lq}
		\left(\sum_{K\in\mesh}|K|\,|\widetilde{v}_K-\varphi(\x_K)|^q\right)^{1/q}\lesssim q h \norm{\varphi}{C^2(\overline{\Omega})}.
		\end{equation}
		We then set $w=(\psi^a(\x_K)\widetilde{v}_K)_{K\in\mesh}$, that belongs to $X_{\disc,0}$ if $h\le a/4$. It is proved in \cite[Proof of Lemma 4.4, p. 2032]{biharmonicFV} that, with $[\Delta \varphi]_K=\frac{1}{|K|}\int_K \Delta\varphi\,dx$,
		\begin{align}
		\Delta_K w - [\Delta \varphi]_K ={}& (\widetilde{v}_K-\varphi(\x_K))\Delta_K \psi^a_\disc + \frac{1}{|K|}\sum_{\sigma\in\mathcal F_K}\frac{|\sigma|}{d_\sigma} (\delta_{K,\sigma}\psi_\disc^a)\delta_{K,\sigma}(\widetilde{v}-\varphi_\disc),\nonumber\\
		={}&T_{1,K}+T_{2,K}.
		\label{est.DeltaDw}
		\end{align}
		Using H\"older's inequality with exponents $(q,\frac{2q}{q-2})$, for some $q>2$ admissible in \eqref{est.Lq}, and recalling \eqref{est.psia}, we have
		\begin{equation}\label{est.T1}
		\left(\sum_{K\in\mesh}|K|\,|T_{1,K}|^2\right)^{1/2}\lesssim q ha^{-2+\frac{q-2}{2q}}\norm{\varphi}{C^2(\overline{\Omega})}.
		\end{equation}
		On the other hand, we have $|\delta_{K,\sigma}\psi_\disc^a|\lesssim d_\sigma a^{-1}$ (see \cite[Proof of Lemma 4.4]{biharmonicFV}). Hence, by Cauchy--Schwarz inequality on the sum over the faces, and using the estimate $\sum_{\sigma\in\mathcal F_K}|\sigma|d_\sigma\lesssim |K|$,
		\[
		|T_{2,K}|^2\lesssim \frac{a^{-2}}{|K|^2}\left(\sum_{\sigma\in\mathcal F_K}|\sigma|\,|\delta_{K,\sigma}(\widetilde{v}-\varphi_\disc)|\right)^2
		\lesssim\frac{a^{-2}}{|K|}\sum_{\sigma\in\mathcal F_K}\frac{|\sigma|}{d_\sigma}(\delta_{K,\sigma}(\widetilde{v}-\varphi_\disc))^2.
		\]
		Estimate \eqref{est.H1} thus leads to
		\begin{equation}\label{est.T2}
		\left(\sum_{K\in\mesh}|K|\,|T_{2,K}|^2\right)^{1/2}\lesssim   a^{-1}h  \norm{\varphi}{C^2(\overline{\Omega})}.
		\end{equation}
		Denote by $[\Delta\varphi]_{\disc}$ the piecewise constant function equal to $[\Delta\varphi]_K$ on $K\in\mesh$. Taking the $L^2$ norm of \eqref{est.DeltaDw} and using \eqref{est.T1} and \eqref{est.T2}, we arrive at, since $a^{-1}\lesssim a^{-\frac{3}{2}-\frac{1}{q}}$,
		\[
		\norm{\Delta_\disc w-[\Delta\varphi]_{\disc}}{L^2(\Omega)}
		\lesssim q h a^{-\frac{3}{2}-\frac{1}{q}}\norm{\varphi}{C^2(\overline{\Omega})}.
		\]
		Taking $q=|\ln(a)|$ if $d=2$ or $q=6$ if $d=3$ shows that 
		\begin{equation}\label{second.cons}
		\norm{\Delta_\disc w-[\Delta\varphi]_{\disc}}{L^2(\Omega)}
		\lesssim h \norm{\varphi}{C^2(\overline{\Omega})}
		\times\left\{\begin{array}{ll}|\ln(a)|a^{-3/2}&\mbox{ if $d=2$},\\
		a^{-5/3}&\mbox{ if $d=3$}.\end{array}\right.
		\end{equation}
		A classical estimate \cite[Lemma B.6]{koala} gives 
		\begin{equation}\label{est.proj}
		\norm{[\Delta\varphi]_\disc-\Delta\varphi}{L^2(\Omega)}\lesssim h\norm{\Delta\varphi}{H^1(\Omega)},
		\end{equation}
		which shows that $\norm{\Delta_\disc w-\Delta\varphi}{L^2(\Omega)}$ is bounded above by the
		right-hand side of \eqref{est.SD.Cc}. The estimates on $\nabla_\disc w-\nabla\varphi$ and on $\Pi_\disc w-\varphi$ follow as in \cite[Lemma 4.4]{biharmonicFV}.

	\smallskip
	$\bullet$ \textsc{Consistency -- general case}:
	Consider now $\varphi\in H^2_0(\Omega)\cap C^2(\overline{\Omega})$, and take $\psi^a$ as above.
		The boundary conditions on $\varphi$ show that $|\varphi(\x)|\lesssim \norm{\varphi}{C^2(\overline{\Omega})}\dist(\x,\partial\Omega)^2$ and $|\nabla\varphi(\x)|\lesssim \norm{\varphi}{C^2(\overline{\Omega})}\dist(\x,\partial\Omega)$. Hence, using \eqref{prop.psia}, $|\O\backslash\O_a|\lesssim a$ and the fact that $1-\psi^a=0$ in $\Omega_a$, we see that, for all $\alpha\in\N^d$ with $|\alpha|\le 2$,
		\begin{equation}\label{est.partial.phi}
		\norm{\partial^\alpha \varphi-\partial^\alpha(\psi^a\varphi)}{L^2(\Omega)}\lesssim a^{1/2} \norm{\varphi}{C^2(\overline{\Omega})}.
		\end{equation}
		Since $\Delta=\sum_{i=1}^2 \partial_i^2$, the above estimate applies to $\Delta$ instead of $\partial^\alpha$ and, as a consequence,
		\begin{equation}\label{est.Deltaphi2}
		\norm{[\Delta \varphi]_\disc-[\Delta(\psi^a\varphi)]_\disc}{L^2(\Omega)}
		\le \norm{\Delta \varphi-\Delta(\psi^a\varphi)}{L^2(\Omega)}
		\lesssim a^{1/2} \norm{\varphi}{C^2(\overline{\Omega})}.
		\end{equation}
		Consider now the interpolant $w\in X_{\disc,0}$ for $\psi^a\varphi\in C^2_c(\Omega)$ constructed 
		above. Applying \eqref{second.cons} to $\psi^a\varphi$ instead of $\varphi$, noting that $\norm{\psi^a\varphi}{C^2(\overline{\Omega})}\lesssim \norm{\varphi}{C^2(\overline{\Omega})}$ (consequence of \eqref{est.partial.phi}), and using \eqref{est.Deltaphi2}, we obtain
		\[
		\norm{\Delta_\disc w-[\Delta\varphi]_{\disc}}{L^2(\Omega)}
		\lesssim a^{1/2} \norm{\varphi}{C^2(\overline{\Omega})}+h \norm{\varphi}{C^2(\overline{\Omega})}
		\times\left\{\begin{array}{ll}|\ln(a)|a^{-3/2}&\mbox{ if $d=2$},\\
		a^{-5/3}&\mbox{ if $d=3$}.\end{array}\right.
		\]
		Taking $a=h^{1/2}$ if $d=2$ or $a=h^{6/13}$ if $d=3$ leads to
		\[
		\norm{\Delta_\disc w-[\Delta\varphi]_{\disc}}{L^2(\Omega)}
		\lesssim \norm{\varphi}{C^2(\overline{\Omega})}\times\left\{\begin{array}{ll}h^{1/4}|\ln(h)|&\mbox{ if $d=2$},\\
		h^{3/13}&\mbox{ if $d=3$}.\end{array}\right.
		\]
Combined with \eqref{est.proj} this shows that $\norm{\Delta_\disc w-\Delta\varphi}{L^2(\Omega)}$ is bounded above by the right-hand side of \eqref{est.SD.FV}. The estimates on $\Pi_\disc w-\varphi$ and $\nabla_\disc w-\nabla \varphi$ follow in a similar way.

	\smallskip
	$\bullet$ \textsc{Limit-conformity}: 
	For $ \xi \in \wdspace(\O)$ and $ v_\disc \in X_{\disc,0}$,  $B=\frac{{\rm tr}(\cdot)}{\sqrt{d}}{\rm Id}$ implies
	$$\int_\O (\hessian:B^{\tau}B\xi)\Pi_\disc v_\disc\d\x =\int_\O(B\hessian:B\xi)\Pi_\disc v_\disc\d\x=\int_{\O}^{}\Delta \phi \Pi_\disc v_\disc \d\x,$$ where $\phi={\rm tr}(\xi).$ Also, by definition of $\hbd$,
	$$\int_\O B\xi:\hbd v_\disc \d\x=\int_{\O}^{}\phi\Delta_\disc v_\disc\d\x.$$ Thus, \eqref{def.WD} can be rewritten as
	\be
	W_\disc^B(\xi)=\max_{v_\disc\in X_{\disc,0}\backslash\{0\}}
	\frac{1}{\norm{\hbd v_\disc}{}}\Bigg|\int_\O \Big(\Delta \phi \Pi_\disc v_\disc - \phi \Delta_\disc v_\disc \Big)\d\x \Bigg|, \label{biharmonic_WD}
	\ee
	where $\phi={\rm tr}(\xi)$. Define
	\begin{equation}
	\hat{\delta}_\edge\phi=
	\begin{cases}
	\phi(\x_{\cell_\edge^+})-\phi(\x_{\cell_\edge^-}) \,\,&\forall \edge\in \edgesint \\
	\phi(\z_\edge)-\phi(\x_{\cell_\edge})\,\,&\forall \edge \in \edgesext,
	\end{cases}\label{deltahat}
	\end{equation} 
	where $\z_\edge$ is the orthogonal projection of $\x_\cell$ on the hyperplane which contains $\edge.$ For $\xi \in H^2(\O)^{d \times d}$, using the divergence theorem, 
	\[
	\int_{\O}^{}\Delta \phi \Pi_\disc v_\disc \d\x  = \sum_{\cell \in \mesh}^{} \int_{\cell}^{}\Delta \phi \Pi_\disc v_\disc \d\x 
	= \sum_{\cell \in \mesh}^{}\sum_{\edge \in \edgescv}^{} v_\cell\int_{\edge}^{}\nabla \phi\cdot n_{\cell,\edge} \d s(\x).
	\]
	Gathering over the edges and using the definition of $\delta_{\edge}$, this leads to
	\begin{align}
	\int_{\O}^{}\Delta \phi{}& \Pi_\disc v_\disc \d\x = -\sum_{\edge \in \edges}^{} \delta_\edge v_\disc\int_{\edge}^{}\nabla \phi\cdot n_{\edge} \d s(\x) \nonumber \\
	& = -\sum_{\edge \in \edges}^{} \delta_\edge v_\disc\int_{\edge}^{} \Big(\frac{\hat{\delta}_\edge \phi}{d_\edge}+\nabla \phi\cdot n_{\edge}- \frac{\hat{\delta}_\edge\phi}{d_\edge} \Big) \d s(\x) \nonumber \\
	& = -\sum_{\edge \in \edges}^{} \delta_\edge v_\disc \frac{\hat{\delta}_\edge \phi |\edge|}{d_\edge} + \sum_{\edge \in \edges}^{} \delta_\edge v_\disc\int_{\edge}^{} \Big(\frac{\hat{\delta}_\edge \phi}{d_\edge}-\nabla \phi \cdot n_{\edge}\Big) \d s(\x). \label{limitconformity_bihar}
	\end{align}
	Since $\delta_{\edge}v_\disc=0$ for any $\edge \in \edgesext$, \eqref{deltahat}, \eqref{jump} and \eqref{biharmonic_delta} imply
	\begin{align*}
	-\sum_{\edge \in \edges}^{} \delta_\edge v_\disc \frac{\hat{\delta}_\edge \phi |\edge|}{d_\edge}&=-\sum_{\edge \in \edgesint}^{}  \frac{|\edge|}{d_\edge}\delta_\edge v_\disc \Big(\phi(\x_{\cell_{\edge}}^+)-\phi(\x_{\cell_{\edge}}^-)\Big) \\
	&=\sum_{\cell \in \mesh}^{}\phi(\x_\cell)\sum_{\edge \in \edgescv}^{}  \frac{|\edge|}{d_\edge}\delta_{\cell,\edge} v_\disc = \sum_{\cell \in \mesh}^{}|\cell|\phi(\x_\cell)\Delta_\cell v_\disc.
	\end{align*} 
	Substituting this in \eqref{limitconformity_bihar}, we obtain
	\begin{equation}\label{FV.limconf.1}
	\begin{aligned}
	\int_{\O}^{}\Delta \phi \Pi_\disc v_\disc \d\x ={}& \sum_{\cell \in \mesh}^{}|\cell|\phi(\x_\cell)\Delta_\cell v_\disc\\
	&+ \sum_{\edge \in \edges}^{} \delta_\edge v_\disc\int_{\edge}^{} \Big(\frac{\hat{\delta}_\edge \phi}{d_\edge}-\nabla \phi\cdot n_{\edge}\Big) \d s(\x).
	\end{aligned}
	\end{equation}
	To deal with the first term, we first combine the two estimates in \cite[Lemma 7.61]{koala} to see that
	\[
	|\phi(\x_K)-\phi(\y)|\le  C h|K|^{-1/2}\norm{\phi}{H^2(K)}\,,\qquad\forall \y\in K.
	\]
	Hence, using the Cauchy--Schwarz inequality,
	\begin{align}
	\Bigg|\sum_{\cell \in \mesh}^{}|\cell|{}&\phi(\x_\cell)\Delta_\cell v_\disc-\int_\O \phi\Delta_\disc v_\disc\d\x\Bigg|\nonumber\\
	={}& \Bigg|\sum_{\cell \in \mesh}^{}|\cell|\left(\phi(\x_\cell)-\frac{1}{|K|}\int_K \phi(\y)\d\y\right)\Delta_\cell v_\disc\Bigg|\nonumber\\
	\le{}& C h\norm{\phi}{H^2(\O)}\left(\sum_{\cell \in \mesh}^{}|\cell||\Delta_\cell v_\disc|^2\right)^{1/2}
	= C h\norm{\phi}{H^2(\O)}\norm{\Delta_\disc v_\disc}{}.
	\label{FV.limconf.2}
	\end{align}
	Turning to the second term in the right-hand side of \eqref{FV.limconf.1}, we notice that the estimate on the terms $R_{K,\edge}$ in \cite[Proof of Theorem 3.4]{EGH00} show that
	\[
	\left|\frac{\hat{\delta}_\edge\phi}{d_\edge}-\nabla \phi\cdot n_{\edge}\right|\le C h \frac{\sqrt{|\edge|}}{\sqrt{d_\edge}}\norm{\hessian \phi}{L^2(\cup_{L\in \mesh_\edge}L)^{d\times d}}.
	\]
	Hence, by the Cauchy--Schwarz inequality, we have
	\begin{align}
	\Bigg|\sum_{\edge \in \edges}^{} \delta_\edge v_\disc\int_{\edge}^{} {}&\Big(\frac{\hat{\delta}_\edge \phi}{d_\edge}-\nabla \phi\cdot n_{\edge}\Big) \d s(\x)\Bigg|
	\le  C h\norm{\hessian \phi}{}\left(\sum_{\edge \in \edges}^{}\frac{|\edge|}{d_\edge}(\delta_\edge v_\disc)^2\right)^{1/2}\nonumber\\
	={}& Ch\norm{\phi}{H^2(\O)}\norm{v_\disc}{\disc} 
	\le  Ch\mbox{diam}(\O)\norm{\phi}{H^2(\O)}\norm{\Delta_\disc v_\disc}{},
	\label{FV.limconf.3}
	\end{align}
	where we have used \eqref{ineq_1} in the last line. Plugging \eqref{FV.limconf.2} and \eqref{FV.limconf.3} into \eqref{FV.limconf.1}, we obtain
	\[
	\left|\int_{\O}^{}\Delta \phi \Pi_\disc v_\disc \d\x-\int_{\O}^{}\phi\Delta_\disc v_\disc\d\x\right|\le Ch\norm{\phi}{H^2(\O)}\norm{\Delta_\disc v_\disc}{},
	\]
	and the estimate on $W_\disc(\xi)$ then follows from \eqref{biharmonic_WD}, recalling that $\phi={\rm tr}(\xi)$.
\end{proof}
\begin{remark}\label{FVM.remark.genericmesh}
	The same analysis also probably applies to the second method presented in \cite[Section 5]{biharmonicFV}, which is applicable on general polygonal meshes.
\end{remark}

\section{Numerical results}\label{sec.example}

In this section, we present the results of some numerical experiments for the gradient recovery (GR) method and finite volume (FV) method presented in Sections \ref{sec.grmethod} and \ref{sec.fvm}.
All these tests are conducted on the biharmonic problem $\Delta^2 \bu=f$ on $\O=(0,1)^2$, with clamped boundary conditions and for various exact solutions $\bu$.

\subsection{Numerical results for Gradient Recovery method}\label{sec.GRexample}

Three examples are presented to illustrate the theoretical estimates of Theorem \ref{error} on the Hessian discretisation described in Section \ref{sec:biorth}. 
The considered FE space $V_h$ is therefore the conforming $\mathbb{P}_1$ space,
and the implementation was done following the ideas in \cite{BL_stab.mixedfem}. The following relative errors,
and related orders of convergence, in $L^2(\O)$, $H^1(\O)$ and $H^2(\O)$ norms are presented:
	\begin{align*}
	&\err_\disc(\bu):=\frac{\norm{\Pi_\disc u_\disc -\bu}{}}{\norm{\bu}{}},\quad
	\err(\nabla\bu) :=\frac{\norm{\nabla u_\disc -\nabla\bu}{}}{\norm{\nabla\bu}{}}\\ 
	&\err_\disc(\nabla\bu) :=\frac{\norm{\nabla_\disc u_\disc -\nabla\bu}{}}{\norm{\nabla\bu}{}}=\frac{\norm{Q_h \nabla u_\disc -\nabla\bu}{}}{\norm{\nabla\bu}{}},\\
&\err_\disc(\hessian\bu) :=\frac{\norm{\hbd u_\disc -\hessian\bu}{}}{\norm{\hessian\bu}{}}=\frac{\norm{\nabla (Q_h \nabla u_\disc) -\hessian\bu}{}}{\norm{\hessian\bu}{}},
	\end{align*}
	where $u_\disc$ is the solution to the Hessian scheme \eqref{base.HS}. 		
											
We provide in Table \ref{table.comparsion1} the mesh data: mesh sizes $h$, numbers of unknowns (that is, the number of internal vertices) $\textbf{nu}$, and numbers of non-zero terms $\textbf{nnz}$ in the square matrix of the system.
		\begin{table}[h!!] 
			\caption{\small{(GR) Mesh size, number of unknowns and number of non-zero terms in the square matrix }}
			{\small{\footnotesize
					\begin{center}
						\begin{tabular}{| |c|c|c||}
							\hline
							$	h$&	$\textbf{nu}$&$\textbf{nnz}$ \\ 
							\hline\\[-10pt]  &&\\[-9pt]
							
							0.176777&9&   79\\
							0.088388&49& 1203\\
							0.044194&225&  7011\\
							0.022097&961&   32835\\
							0.011049&3969&  141315\\
							0.005524&16129&585603\\
							\hline				
						\end{tabular}
					\end{center}}	}\label{table.comparsion1}
				\end{table}

	\subsubsection{Example 1}
		The exact solution is chosen to be $\bu(x,y) = x^2(x-1)^2y^2(y-1)^2$. To assess the effect of the stabilisation function $\stab_h$ on the results, we multiply it by a factor $r$ that takes the values 0.1, 1, 10, and 100.
		
The errors and orders of convergence for the numerical approximation to $\bu$ are shown in Tables \ref{table1}--\ref{table4}. It can be seen that the rate of convergence is quadratic in $L^2$-norm and linear in $H^1$-norm (see $\err(\nabla\bu)$). However, using gradient recovery operator, a quadratic order of convergence in $H^1$ norm is recovered (see $\err_\disc(\nabla \bu)$).
The rate of convergence in energy norm is linear (see $\err_\disc(\hessian\bu)$), as expected by plugging the estimates of Theorem \ref{thm:HD.coerciveB} into Theorem \ref{error}. We also notice a very small effect of $r$ on the relative errors and rates.

								\begin{table}[h!!]
							\caption{\small{(GR) Convergence results for the relative errors, Example 1, $r=0.1$}}
							{\small{\footnotesize
									\begin{center}
										\begin{tabular}{ |c|c|c||c | c ||c|c|| c|c||c|c||c|c||}
											\hline
$\textbf{nu}$ &$\err_\disc(\bu)$ & Order  & $\err(\nabla \bu)$ & Order& $\err_\disc(\nabla \bu)$ & Order  &$\err_\disc(\hessian \bu)$ & Order  \\ 
											\hline\\[-10pt]  &&\\[-9pt]
9&   9.274702&        -& 31.591906& -&  0.568338&       -&    0.595635&              -\\
49&   0.220095&   5.3971&  0.682922&  5.5317&  0.164105&   1.7921&   0.266927&   1.1580\\
225&   0.066997&   1.7160&  0.201282&   1.7625& 0.049395&   1.7322&   0.128410&   1.0557\\
961&   0.019135&   1.8079&  0.088805&   1.1805&  0.013697&   1.8505&   0.062164&   1.0466\\
3969&   0.005133&   1.8983&  0.040845&   1.1205&  0.003623&   1.9185&   0.030457&   1.0293\\ 	
16129&0.001331& 1.9474&   0.019422&   1.0724&0.000933& 1.9568& 0.015059& 1.0161\\				
											\hline				
										\end{tabular}
									\end{center}	}}\label{table1}
								\end{table}		

		\begin{table}[h!!] 
			\caption{\small{(GR) Convergence results for the relative errors, Example 1, $r=1$}}
			{\small{\footnotesize
					\begin{center}
						\begin{tabular}{ |c||c |c||c| c|| c|c||c|c||c|c|}
							\hline
	$\textbf{nu}$&$\err_\disc(\bu)$ & Order  &$\err(\nabla \bu)$ & Order&$\err_\disc(\nabla \bu)$ & Order  &$\err_\disc(\hessian \bu)$ & Order  \\ 
							\hline\\[-10pt]  &&\\[-9pt]
							
9& 1.050930&        - &  3.254044&        -& 0.567670 &        -&  0.582647&      -\\
49&  0.214195&   2.2947&   0.482686&   2.7531& 0.167145&   1.7640&    0.267188&  1.1248\\
225& 0.067498&   1.6660&   0.200108&   1.2703& 0.049952&   1.7425&    0.128511&   1.0560\\
961& 0.019240&   1.8107&   0.088667&   1.1743& 0.013806&   1.8553&    0.062184&   1.0473\\
3969& 0.005156&   1.8999&   0.040835&   1.1186& 0.003646&   1.9209&    0.030460&   1.0296\\
16129&0.001336&  1.9482&  0.019421&   1.0722&0.000938& 1.9581& 0.015060&  1.0162\\
							\hline				
						\end{tabular}
					\end{center}}	}\label{table2}
				\end{table}

				\begin{table}[h!!]
					\caption{\small{(GR) Convergence results for the relative errors, Example 1, $r=10$}}
					{\small{\footnotesize
							\begin{center}
								\begin{tabular}{ |c||c | c ||c|c|| c|c||c|c||c|c||}
									\hline
$\textbf{nu}$ &$\err_\disc(\bu)$ & Order  &$\err(\nabla \bu)$ & Order&$\err_\disc(\nabla \bu)$ & Order  &$\err_\disc(\hessian \bu)$ & Order  \\ 
									\hline\\[-10pt]  &&\\[-9pt]
9&   0.661894&     -&    0.778521&    -&   0.583641&       -&    0.586174&      -\\
49&   0.236529&   1.4846& 0.449484&   0.7925&  0.195127&   1.5807&   0.274030&   1.0970\\
225&   0.072610&   1.7038& 0.197892&   1.1836&  0.055493&   1.8140&   0.129911&   1.0768\\
961&   0.020303&   1.8385& 0.088413&   1.1624&  0.014907&   1.8963&   0.062418&   1.0575\\
3969&   0.005382&   1.9154& 0.040804&   1.1156&  0.003877&   1.9429&   0.030494&   1.0335\\		
16129&   0.001387& 1.9564&  0.019417& 1.0714&   0.000990&  1.9695&  0.015064&  1.0174\\			
									\hline				
								\end{tabular}
							\end{center}	}}\label{table3}
						\end{table}		
								\begin{table}[h!!]
									\caption{\small{(GR) Convergence results for the relative errors, Example 1, $r=100$}}
									{\small{\footnotesize
											\begin{center}
												\begin{tabular}{ |c||c | c||c|c || c|c||c|c||c|c||}
													\hline
$\textbf{nu}$ &$\err_\disc(\bu)$ & Order  & $\err(\nabla \bu)$ & Order &$\err_\disc(\nabla \bu)$ & Order  &$\err_\disc(\hessian \bu)$ & Order  \\ 
													\hline\\[-10pt]  &&\\[-9pt]
9&   0.784444&        -&  0.805690&  -& 0.701021&        -&   0.695247&      -\\
49&   0.409420&   0.9381&  0.456340&  0.8201& 0.386868&   0.8576&   0.408281&   0.7680\\
225&   0.123166&   1.7330&  0.199370&  1.1947& 0.108498&   1.8342&   0.157333&   1.3757\\
961&   0.031509&   1.9667&  0.088447&   1.1726& 0.026358&   2.0414&   0.066443&   1.2436\\
3969&   0.007812&   2.0121&  0.040790&   1.1166&  0.006356&   2.0521&   0.031019&   1.0990\\				
16129&0.001934& 2.0139& 0.019414&  1.0711& 0.001552& 2.0340& 0.015130& 1.0357\\		
													\hline				
												\end{tabular}
											\end{center}	}}\label{table4}
										\end{table}		

\subsubsection{Example 2}
We consider here the transcendental exact solution $\bu= x^2(x-1)^2y^2(y-1)^2(\cos(2\pi x)+\sin(2\pi y))$, and $r=0.1, 1$ and $10$. Tables \ref{table5}--\ref{table7} presents the numerical results. The same comments as in Example 1 can be made about the rates of convergence. Past the coarsest meshes, we also notice as in Example 1 that $r$ only has a small impact on the relative errors.

\begin{table}[h!!]
\caption{\small{(GR) Convergence results for the relative errors, Example 2, $r=0.1$}}
	{\small{\footnotesize
		\begin{center}
								\begin{tabular}{ |c||c | c ||c|c|| c|c||c|c||c|c||}
			\hline
$\textbf{nu}$&$\err_\disc(\bu)$ & Order&$\err(\nabla \bu)$ & Order  &$\err_\disc(\nabla \bu)$ & Order&$\err_\disc(\hessian \bu)$ & Order  \\   
\hline\\[-10pt]  &&\\[-9pt]
9&89.040689& -& 183.461721& -&   1.211097&  -& 1.614525&  -\\
49&0.825060& 6.7538&3.401374&5.7532&0.235295& 2.3638&  0.501568& 1.6866\\
225&0.076841& 3.4246&0.337917&3.3314&0.050832& 2.2107&  0.172310& 1.5414\\
961&0.017830& 2.1076&0.114315&1.5637&0.013579& 1.9044&  0.079638& 1.1135\\
3969&0.004565& 1.9655&0.052228&1.1301&0.003638& 1.9002&  0.039166&  1.0239\\
16129&  0.001168& 1.9662& 0.025518& 1.0333& 0.000949&1.9391& 0.019457&  1.0093\\
																			\hline				
						\end{tabular}
	\end{center}	}}\label{table5}
\end{table}			
										
\begin{table}[h!!]
\caption{\small{(GR) Convergence results for the relative errors, Example 2, $r=1$}}
{\small{\footnotesize
\begin{center}
								\begin{tabular}{ |c||c | c ||c|c|| c|c||c|c||c|c||}
		\hline
$\textbf{nu}$ &$\err_\disc(\bu)$ & Order&$\err(\nabla \bu)$ & Order  &$\err_\disc(\nabla \bu)$ & Order&$\err_\disc(\hessian \bu)$ & Order  \\   
\hline\\[-10pt]  &&\\[-9pt]
9&10.222667&      -&  19.376883&       -&  1.058048 &    -& 1.333720&       -\\
49&0.475973&   4.4247&1.467316& 3.7231&   0.229176& 2.2069&   0.473233& 1.4948\\
225&0.074399&   2.6775&0.313397& 2.2271&   0.050755& 2.1748&   0.170477& 1.4730\\
961&0.017711&   2.0706&0.112806& 1.4742&   0.013591& 1.9009&   0.079552& 1.0996\\
3969&0.004547&   1.9615&0.052162& 1.1128&   0.003640& 1.9006&   0.039162& 1.0224\\
16129& 0.001164& 1.9657& 0.025515& 1.0317&   0.000949& 1.9393&   0.019456&  1.0092\\
											\hline				
														\end{tabular}
													\end{center}	}}\label{table6}
												\end{table}		

	\begin{table}[h!!]
	\caption{\small{(GR) Convergence results for the relative errors, Example 2, $r=10$}}
	{\small{\footnotesize
		\begin{center}
								\begin{tabular}{ |c||c | c ||c|c|| c|c||c|c||c|c||}
			\hline
$\textbf{nu}$ &$\err_\disc(\bu)$ & Order&$\err(\nabla \bu)$ & Order  &$\err_\disc(\nabla \bu)$ & Order&$\err_\disc(\hessian \bu)$ & Order  \\   
	\hline\\[-10pt]  &&\\[-9pt]
9&1.413122&  -& 2.541143&     -&  0.845365&   -&   0.894504&     -\\
49&0.313425& 2.1727&   0.878752& 1.5319&0.225247&  1.9081& 0.396725&  1.1729\\
225&0.066842& 2.2293&   0.262354& 1.7439&0.051757&  2.1217& 0.165546& 1.2609\\
961&0.016897& 1.9840&   0.109794& 1.2567&0.013783&  1.9089& 0.079311& 1.0616\\
3969&0.004376& 1.9492&   0.052012& 1.0779&0.003675&  1.9072& 0.039149&  1.0185\\		
16129&	   0.001123& 1.9621& 0.025506&    1.0280 &    0.000956& 1.9425& 0.019455&  1.0088\\														
																	\hline				
																\end{tabular}
															\end{center}	}}\label{table7}
														\end{table}

\subsubsection{Example 3}

{Here, $\bu(x,y)= x^3y^3(1-x)^3(1-y)^3(e^x\sin(2\pi x)+\cos(2\pi x))$ and $r=0.1, 1$ and $10$. The results presented in Tables \ref{table8}--\ref{table10} are similar to those obtained for Examples 1 and 2.

\begin{table}[h!!]
\caption{\small{(GR) Convergence results for the relative errors, Example 3, $r=0.1$}}
	{\small{\footnotesize																				\begin{center}
								\begin{tabular}{ |c||c | c ||c|c|| c|c||c|c||c|c||}
			\hline
$\textbf{nu}$&$\err_\disc(\bu)$ & Order &$\err(\nabla \bu)$ & Order &$\err_\disc(\nabla \bu)$ & Order  &$\err_\disc(\hessian \bu)$ & Order  \\ 
																										\hline\\[-10pt]  &&\\[-9pt]
																							9&   81.804173&       -&  164.358300&      -&   1.068682&       -&  1.155266&      -\\
49&   0.677743&  6.9153&   2.358209&  6.1230&   0.232374&    2.2013&   0.517095&    1.1597\\
225&   0.093340&  2.8602&   0.447143&  2.3989&   0.048701&    2.2544&   0.207642&    1.3163\\
961&   0.017130&  2.4459&   0.125296&  1.8354&   0.010361&    2.2328&    0.084719&   1.2933\\
3969&   0.003975&  2.1074&   0.053941&  1.2159&   0.002643&    1.9711&    0.041197&   1.0401\\
16129&0.000982& 2.0167& 0.026457&  1.0278&0.000692& 1.9341&  0.020529& 1.0049\\																									\hline				
																									\end{tabular}
																								\end{center}	}} \label{table8}
																							\end{table}	

\begin{table}[h!!]
\caption{\small{(GR) Convergence results for the relative errors, Example 3, $r=1$}}
		{\small{\footnotesize
				\begin{center}
								\begin{tabular}{ |c||c | c ||c|c|| c|c||c|c||c|c||}
			\hline
$\textbf{nu}$&$\err_\disc(\bu)$ & Order &$\err(\nabla \bu)$ & Order &$\err_\disc(\nabla \bu)$ & Order  &$\err_\disc(\hessian \bu)$ & Order  \\ 
	\hline\\[-10pt]  &&\\[-9pt]
9&8.708395& -    &16.990965&   -& 0.950590&       -&   0.990455&    - \\
49&0.516904&4.0744&1.490046&3.5113&0.224877&  2.0797&   0.492555& 1.0078\\
225&0.089332&2.5326&0.414243&1.8468&0.048056&  2.2263&   0.203301& 1.2767\\
961&0.016920&2.4005&0.122315&1.7599&0.010349&  2.2153&   0.084441&  1.2676\\
3969&0.003953&2.0975&0.053813&1.1846&0.002646&  1.9678&   0.041186&  1.0358\\
16129&   0.000978& 2.0153&  0.026452&1.0246&   0.000693& 1.9337& 0.020528& 1.0045\\
\hline				
	\end{tabular}
	\end{center}	}}\label{table9}
		\end{table}}

\begin{table}[h!!]
\caption{\small{(GR) Convergence results for the relative errors, Example 3, $r=10$}}
		{\small{\footnotesize
		\begin{center}
								\begin{tabular}{ |c||c | c ||c|c|| c|c||c|c||c|c||}
																								\hline
																							$\textbf{nu}$ &$\err_\disc(\bu)$ & Order &$\err(\nabla \bu)$ & Order &$\err_\disc(\nabla \bu)$ & Order  &$\err_\disc(\hessian \bu)$ & Order  \\ 
																								\hline\\[-10pt]  &&\\[-9pt]
																							9& 1.097695&     -&  2.068091&  -  &   0.809189&   -&  0.792818&   -\\
																							49& 0.351280&1.6438&  0.969172&1.0935&  0.205661& 1.9762& 0.409436& 0.9533\\
																							225& 0.073936&2.2483&  0.306858&1.6592&  0.046151& 2.1558& 0.186959& 1.1309\\
																							961& 0.015689&2.2365&  0.113622&1.4333&  0.010414& 2.1478& 0.083455& 1.1637\\
																							3969& 0.003756&2.0624&  0.053444&1.0882&  0.002689& 1.9535& 0.041142& 1.0204\\
	16129&   0.000935&2.0068& 0.026437& 1.0155&   0.000705& 1.9309& 0.020526&  1.0032\\																							\hline				
																							\end{tabular}
	\end{center}	}} \label{table10}
		\end{table}		
										
\subsection{Numerical results for FVM}\label{sec.FVMexample}
	In this section, we present numerical results based on the finite volume method presented in Section \ref{sec.fvm}. As noticed, this scheme requires only one unknown per cell, and is therefore easy to implement and computationally cheap. The schemes were first tested on a series of regular triangular meshes (\texttt{mesh1} family) and then on square meshes (\texttt{mesh2} family), both taken from \cite{benchmark}. To ensure the correct orthogonality property (see Definition \ref{FV.def}), the point $\x_\cell \in \cell$ is chosen as the circumcenter of $K$ if $K$ is a triangle, or the center of mass of $K$ if $K$ is a rectangle. As a result, for triangular meshes, the $L^2$ error, $\err_\disc(\bu)$, is calculated using a skewed midpoint rule, where we consider the circumcenter of each cell instead of its center of mass. We denote the relative $H^2$ error by
	\[
	\err_\disc(\Delta\bu) :=\frac{\norm{\Delta_\disc \bu_\disc -\Delta \bu}{}}{\norm{\Delta\bu}{}}.
	\]
	The $H^1$ and $H^2$ errors $(\err_\disc(\nabla\bu)$ and $\err_\disc(\Delta\bu))$ are computed using the usual midpoint rule. For comparsion with the gradient recovery method (see Table \ref{table.comparsion1}), the details of mesh size $h$, number of unknowns $\textbf{nu}$ and the number of non-zero terms in the system square matrix $\textbf{nnz}$ for the finite volume method are also provided in the following tables.
	
	\subsubsection{Example 1}\label{ex1}
	In the first example, we choose the right hand side load function $f$ such that the exact solution is given
	by $\bu(x,y) = x^2y^2(1-x)^2(1-y)^2$.  Tables \ref{ex1.triangle} and \ref{ex1.square}
	show the  relative errors and order of convergence rates for the variable $\bu_\disc$ on triangular and square grids. As seen in the table, we obtain linear (in $H^1$-like norm) and sub-linear convergence rates (in $H^2$-like norm) for triangular grids, and quadratic order of convergence for square grids. This behaviour has already been observed in \cite{biharmonicFV}. With respect to $L^2$ norm, quadratic (or slightly better) order of convergence is obtained. These numerical order of convergence are better than the orders of convergences from the theoretical analysis, see Remark \ref{FVM.remark}. This is somehow expected as, due to the difficulty of finding a proper interpolant for this very low-order method \cite{biharmonicFV}, the theoretical rates are much below than the actual rates.
	
	\begin{table}[h!!]
		\caption{\small{(FV) Convergence results, Example 1, triangular grids (\texttt{mesh1} family)}}
		{\small{\footnotesize
				\begin{center}
					\begin{tabular}{ ||c|c|c||c| c||c| c ||c|c||}
						\hline
	$h$ &$\textbf{nu}$&$\textbf{nnz}$ &$\err_\disc(\bu)$ & Order &$\err_\disc(\nabla\bu)$ & Order  &$\err_\disc(\Delta\bu)$ & Order  \\ 
						\hline\\[-10pt]  &&\\[-9pt]
0.250000& 56 &   392& 0.137345&           -&   0.256342&      -&    0.162222&            -\\
0.125000& 224&  1896&  0.031705&   2.1150&   0.131915&   0.9585&   0.071457&   1.1828\\
0.062500& 896&   8264&  0.007400&   2.0991&   0.066136&   0.9961&   0.038596&   0.8886\\
0.031250& 3584&  34440&  0.001691&   2.1297&   0.033067&   1.0000&   0.022662&   0.7682\\
0.015625& 14336& 140552&  0.000352&   2.2644&   0.016528&   1.0005&   0.014158&   0.6786\\
0.007813& 57344& 567816&  0.000056&   2.6449&   0.008262&   1.0004&   0.009281&   0.6092\\
						
						\hline				
					\end{tabular}
				\end{center}	}} \label{ex1.triangle}
	\end{table}			
			
	\begin{table}[h!!]
	\caption{\small{(FV) Convergence results, Example 1, square grids (\texttt{mesh2} family)}}
				{\small{\footnotesize
						\begin{center}
							\begin{tabular}{ ||c|c|c||c| c||c| c ||c|c|| }
								\hline
$h$ &$\textbf{nu}$&$\textbf{nnz}$ &$\err_\disc(\bu)$ & Order &$\err_\disc(\nabla\bu)$ & Order  &$\err_\disc(\Delta\bu)$ & Order  \\ 
								\hline\\[-10pt]  &&\\[-9pt]
0.353553& 16  & 56 & 0.328639&       -&  0.417244&      - &   0.260189&     -\\
0.176777& 64 & 472&  0.081325&   2.0147&   0.107484&   1.9568&   0.062624&  2.0548\\
0.088388& 256&2552&  0.020161&   2.0121&   0.026808&   2.0034&   0.015430&  2.0210\\
0.044194& 1024& 11704&  0.005028&   2.0035&   0.006694&   2.0018&   0.003842&  2.0057\\
0.022097& 4096& 49976&  0.001256&   2.0009&   0.001673&   2.0005&   0.000960&   2.0015\\
0.011049&  16384& 206392& 0.000314&   2.0002&   0.000418&   2.0001&   0.000240&   2.0004\\
								\hline				
							\end{tabular}
			\end{center}	}}\label{ex1.square}
	\end{table}
					
\subsubsection{Example 2}\label{ex2}
In this example, we perform the numerical experiment for the exact solution given by $\bu(x,y) = x^2y^2(1-x)^2(1-y)^2(\cos(2\pi x)+\sin(2\pi y))$. The errors in the energy norm, $H^1$ norm and the $L^2$ norm, together with their orders of convergence, are presented in Tables \ref{ex2.triangle} and \ref{ex2.square}. The results are similar to those for Example 1.
					
\begin{table}[h!!]
	\caption{\small{(FV) Convergence results, Example 2, triangular grids (\texttt{mesh1} family)}}
				{\small{\footnotesize
	\begin{center}
		\begin{tabular}{ ||c|c|c||c| c||c| c ||c|c|| }
										\hline
$h$&$\textbf{nu}$&$\textbf{nnz}$  &$\err_\disc(\bu)$ & Order &$\err_\disc(\nabla\bu)$ & Order  &$\err_\disc(\Delta\bu)$ & Order  \\ 
			\hline\\[-10pt]  &&\\[-9pt]
0.250000&   56  &392&  0.418276&      -&     0.533799&     -&  0.274105&             -\\
0.125000&  224  &1896&  0.075761&   2.4649&   0.204870&   1.3816&   0.101375&   1.4350\\
0.062500&   896  &8264& 0.013663&   2.4712&   0.093729&   1.1281&   0.044254&   1.1958\\
0.031250&  3584 & 34440& 0.003218&   2.0862&   0.046056&   1.0251&   0.021933&   1.0127\\
0.015625& 14336  & 140552&  0.000784&   2.0365&   0.022932&   1.0060& 0.011500&   0.9315\\
0.007813& 57344& 567816&  0.000191&   2.0414&   0.011454&   1.0015&   0.006323&   0.8630\\
										\hline				
			\end{tabular}
		\end{center}	}}\label{ex2.triangle}
\end{table}
							
	\begin{table}[h!!]
		\caption{\small{(FV) Convergence results, Example 2, square grids (\texttt{mesh2} family)}}
				{\small{\footnotesize
			\begin{center}
		\begin{tabular}{ ||c|c|c||c| c||c| c ||c|c||}
						\hline
$h$&$\textbf{nu}$&$\textbf{nnz}$  &$\err_\disc(\bu)$ & Order &$\err_\disc(\nabla\bu)$ & Order  &$\err_\disc(\Delta\bu)$ & Order  \\ 	
			\hline\\[-10pt]  &&\\[-9pt]
0.353553&    16 &   56&  1.333981&         -&   0.745194&     -&   0.773521&      -\\
0.176777&  64&  472&  0.223384&   2.5781&   0.135128&   2.4633&   0.175192&   2.1425\\
0.088388& 256 & 2552&  0.050527&   2.1444&   0.030239&   2.1599&   0.042123&   2.0563\\
0.044194& 1024& 11704&  0.012331&   2.0347&   0.007339&   2.0427&   0.010416&   2.0158\\
0.022097& 4096& 49976&  0.003065&   2.0086&   0.001821&   2.0109&   0.002597&   2.0041\\
0.011049& 16384& 206392& 0.000765&   2.0021&   0.000454&   2.0027&   0.000649&   2.0010\\
												\hline				
						\end{tabular}
	\end{center}	}}\label{ex2.square}
	\end{table}

\subsubsection{Example 3}\label{ex3}
The numerical results obtained  for $\bu(x,y) =  x^3y^3(1-x)^3(1-y)^3(\exp(x)\sin(2\pi x)+\cos(2\pi x))$ are shown in Tables \ref{ex3.triangle} and \ref{ex3.square} respectively. As in Examples 1 and 2, the theoretical rates of convergence are confirmed by these numerical outputs, except that on this test a real linear order of convergence is attained in the $H^2$-like norm.
\begin{table}[h!!]
	\caption{\small{(FV) Convergence results, Example 3, triangular grids (\texttt{mesh1} family)}}
		{\small{\footnotesize
	\begin{center}
\begin{tabular}{ ||c|c|c||c| c||c| c ||c|c||}
		\hline
$h$&$\textbf{nu}$&$\textbf{nnz}$  &$\err_\disc(\bu)$ & Order &$\err_\disc(\nabla\bu)$ & Order  &$\err_\disc(\Delta\bu)$ & Order  \\ 
	\hline\\[-10pt]  &&\\[-9pt]
0.250000&   56&  392&  0.637895&         -&    0.825992&      -&  0.423933&    -\\
0.125000&  224&1896&  0.050763&   3.6515&   0.220328&   1.9065&   0.096604&   2.1337\\
0.062500&896&   8264&   0.013330&   1.9291&   0.097939&   1.1697&   0.045854&   1.0750\\
0.031250&3584& 34440&   0.003160&   2.0765&   0.047945&   1.0305&   0.021417&   1.0983\\
0.015625& 14336& 140552&  0.000786&   2.0084&   0.023857&   1.0070&   0.010550&   1.0215\\
0.007813& 57344& 567816&  0.000196&   2.0016&   0.011914&   1.0017&   0.005257&   1.0049\\
						\hline				
					\end{tabular}
				\end{center}	}}\label{ex3.triangle}
\end{table}						
\begin{table}[h!!]
\caption{\small{(FV) Convergence results, Example 3, square grids (\texttt{mesh2} family)}}
	{\small{\footnotesize
	\begin{center}
		\begin{tabular}{ ||c|c|c||c| c||c| c ||c|c|| }
				\hline
$h$ &$\textbf{nu}$&$\textbf{nnz}$ &$\err_\disc(\bu)$ & Order &$\err_\disc(\nabla\bu)$ & Order  &$\err_\disc(\Delta\bu)$ & Order  \\ 
			\hline\\[-10pt]  &&\\[-9pt]
0.353553&  16&  56&  2.478402&         -&  1.405462&       -&  1.140625&     -\\
0.176777&  64& 472& 0.242959&   3.3506&   0.113945&   3.6246&   0.196693&   2.5358\\
0.088388&  256& 2552& 0.050784&   2.2583&   0.022495&   2.3406&   0.049149&   2.0007\\
0.044194& 1024& 11704&  0.012212&   2.0561&   0.005577&   2.0120&   0.012217&   2.0083\\
0.022097&  4096& 49976&  0.003025&   2.0133&   0.001396&   1.9982&   0.003049&   2.0026\\
0.011049& 16384& 206392& 0.000755&   2.0033&   0.000349&   1.9993&   0.000762&   2.0007\\
				\hline				
				\end{tabular}
			\end{center}	}}\label{ex3.square}
\end{table}
\medskip

Comparing Table \ref{table.comparsion1} and the Tables for FV, we see that the GR method based on biorthogonal reconstruction has only few unknowns (number of internal vertices) but leads to a large stencil for each of them whereas the FV has more unknowns (number of cells) but produces a much sparser matrix. Looking for example at the finest GR mesh and the finest triangular FV mesh, we notice that the meshes have similar sizes $h$ and the matrices have similar complexity \textbf{nnz}, but the FV accuracy in $L^2$- and $H^2$-like norms is much better than the GR method; this is expected since the FV method has a number of unknowns \textbf{nu} more than 3.5 times larger than that of GR. However, the super-convergence property of the gradient reconstruction gives a clear advantage to GR for the $H^1$-like norm. For a similar number of unknowns \textbf{nu} (which means a matrix that is much cheaper to solve for the FV method than the GR method, due to a reduced \textbf{nnz}), the FV method still has a clear advantage in the $L^2$ norm over the GR method, but similar accuracy in the $H^2$-like norm (compare the results for the 5th mesh in the \texttt{mesh1} family with the finest mesh used for the GR method); the GR method however still preserves a clear lead on the $H^1$-like norm error.

\section{Classical FE schemes fitting into the HDM}\label{sec.fem}

We show here that some known FE schemes fit into the Hessian discretisation method, that is, they are Hessian schemes for particular choices of Hessian discretisations.

\subsection{Conforming methods}\label{sec:conf}

For conforming finite elements, we require our finite element space $V_h$ to be a subspace
of the underlying Hilbert space $H^2_0(\O)$. We can then define a Hessian discretisation by $X_{\disc,0}=V_h$ and, for $v\in X_{\disc,0}$, $\Pi_\disc v=v$, $\nabla_\disc v=\nabla v$ and $\hbd v=\hb v$. The estimates on $C_\disc^B$,
$S_\disc^B$ and $W_\disc^B$ easily follow:
\begin{itemize}
	\item $C_\disc^B$ is bounded by the constant of the continuous Poincar\'e inequality in $H^2_0(\O)$.
	\item Standard approximation properties (see, e.g., \cite{ciarlet1978finite}) yield, for almost-affine families of FE, estimates on the interpolation error $S_\disc^B$.
	\item Integration-by-parts in $H^2_0(\O)$ shows that $W_\disc^B(\xi)=0$ for all $\xi \in \wdspace(\O)$.
\end{itemize}

We briefly describe hereafter three finite elements which meet this requirement. The reader is referred to \cite{ciarlet1978finite} for details.

\smallskip

\emph{The Argyris triangle :}
The Argyris triangle is a $C^1$ element which uses a complete
polynomial of degree five. The degrees of freedom consist of function values and
first and second derivatives at the vertices in addition to normal derivatives at the
midpoints of the sides. One difficulty with the Argyris triangle is that there are 21 degrees of freedom per triangle. A modification to the Argyris triangle is the Bell's element which suppresses the values of the normal slopes at the nodes at the three midpoint sides, reducing the number of degrees of freedom to 18 per element.
\smallskip

\emph{Hsieh-Clough-Toucher triangles :}
In the Hsieh-Clough-Tocher (HCT) triangle, the triangle is first decomposed into three triangles by connecting the barycenter of the given triangle with each of its vertices. On each of the subtriangles a cubic polynomial is constructed so that the resulting function is $C^1$ on the original triangle. There are a total of 12 degrees of freedom per triangle, which consist of the function values and first partial derivatives at the three vertices of the original triangle in addition to the normal derivative at the midpoints of the sides of the original triangle.
\medskip

\subsection{An example of non-conforming method: the Adini rectangle}
Assume that $\O$ can be covered by mesh $\mesh$ made up of rectangles (we restrict the presentation to $d=2$ for simplicity). The element $\cell$ consists of a rectangle with vertices $\{a_i, 1 \le i \le 4\}$;
the space $\mathbb{P}_\cell$ is given by $\mathbb{P}_\cell = \mathbb{P}_3 \oplus \{x_1x^3_2\} \oplus\{x^3_1x_2\}$, by which we mean polynomials of degree $\le$ 4 whose only fourth-degree terms are those
involving $x_1x^3_2$
and $x^3_1x_2$. Thus $\mathbb{P}_3 \subset \mathbb{P}_\cell$. The set of degrees of
freedom in each cell is
$$\Sigma_\cell=\bigg\{p(a_i), \frac{\partial p}{\partial x_1}(a_i), \frac{\partial p}{\partial x_2}(a_i); 1\le i \le 4\,,\;p \in \mathbb{P}_\cell\bigg\}.$$
The global approximation space is then given by
\begin{align*}
V_h & =: \{v_h \in L^2(\O);\,\, v_h\rvert_\cell \in \mathbb{P}_\cell \,\forall \,\cell \in \mesh, v_h \mbox{ and }\nabla v_h \mbox{ are continuous at }\\ &\qquad\mbox{ the vertices of elements in } \mesh, v_h \mbox{ and } \nabla v_h \mbox{ vanish at vertices on }\partial \O\}.
\end{align*}
Note that $V_h \subset H^1_0(\O) \cap C^0(\overline{\O})$.

\begin{definition}[Hessian discretisation for the Adini rectangle]\label{def.HD.Adinis}
	Each $v_\disc\in X_{\disc,0}$  is a vector of three values at each vertex of the mesh (with zero values at boundary vertices), corresponding to function and gradient values, $\Pi_\disc v_\disc$ is the function such that $(\Pi_\disc v_\disc)_{|K}\in \mathbb{P}_\cell$ and its gradient takes the values at the vertices dictated by $v_\disc$, $\nabla_\disc v_\disc=\nabla(\Pi_\disc v_\disc)$ and $\hbd v_\disc= \hb_{\mesh}(\Pi_\disc v_\disc)$ is the broken $\hb$ ($\hd$ is the broken $\hessian$).
\end{definition}

We assume that the mesh is regular, that is, \eqref{reg:mesh} holds with $\eta$ not depending on the mesh.

\begin{theorem}\label{thm:HD.Adinis}
	Let $\disc$ be a $B\textendash$Hessian discretisation in the sense of Definition \ref{def.HD.Adinis} with $B$ satisfying the coercive property. Then, there exists a constant $C$, not depending on $\disc$, such that
	\begin{itemize}
		\item  $C_\disc^B \le C$,
		\item  $\forall\varphi \in H^3(\O)\cap H^2_0(\O)$, $S_{\disc}^B(\varphi) \le Ch\norm{\varphi}{H^3(\O)}$,
		\item $\forall \xi \in H^2(\O)^{d \times d}$, $W_{\disc}^B(\xi) \le Ch \norm{\xi}{H^2(\O)^{d \times d}}.$
	\end{itemize}
\end{theorem}

The properties of Hessian discretisations built on the Adini rectangle follow from
this theorem and Remark \ref{rem:dens}.

\begin{corollary}\label{cor:HD.Adinis}
	Let $(\disc_m)_{m \in \N}$ be a sequence of $B\textendash$Hessian discretisations built on the Adini rectangle, such that $B$ is coercive and the underlying sequence of meshes are regular 
and have a size that goes to $0$ as $m\to\infty$. Then the sequence $(\disc_m)_{m \in \N}$ is coercive, consistent and limit-conforming.	
\end{corollary}

\begin{proof}[Proof of Theorem \ref{thm:HD.Adinis}]\textbf{}
	
In this proof, $C > 0$ denotes a generic constant that can change from one line to the other but depends only on $\O$, $d$, $B$ and $\eta$.

\smallskip

$\bullet$ \textsc{Coercivity:} Since $V_h \subset H^1_0(\O),\,$ for $v \in X_{\disc,0}$, the Poincar\'e inequality yields $\norm{\Pi_\disc v}{} \le \mbox{diam}(\O)\norm{\nabla_\disc v}{}$, which gives us part of the estimate on $C_\disc^B$. Define the \emph{broken Sobolev} space
$$H^1(\mesh)=\big\{v \in L^2(\O) \,; \,\forall \cell \in \mesh, v_{|\cell} \in H^1(\cell)\big\}$$
and endow it with the dG norm
\be  \label{dg}
\norm{w}{dG}^2:=\norm{\nabla_\mesh w}{}^2 + \sum_{\edge \in \edges}^{}\frac{1}{h_\edge}\norm{\llbracket w \rrbracket}{L^2(\edge)}^2,
\ee
where 
\begin{equation*}
h_\sigma=\left\{\begin{array}{ll}
\min(h_\cell,h_L)& \mbox{ if }\edge \in \edgesint ,\, \mesh_{\edge}= \{\cell, L\}  \\
h_\cell & \mbox{ if }\edge \in \edgesext,\,\mesh_{\edge}= \cell,
\end{array}\right.
\end{equation*}
and the jump of $w$ is
\[
\llbracket w \rrbracket=
\left\{\begin{array}{ll}
w_{|\cell}- w_{|L}&\mbox{ if } \edge \in \edgesint ,\, \mesh_{\edge}= \{\cell, L\}  \\
w_{|\cell} &\mbox{ if }\edge \in \edgesext,\,\mesh_{\edge}= \cell.
\end{array}\right.
\]
If $\llbracket w \rrbracket = 0$ at the vertices of $\edge$ then, by the Poincar\'e inequality in $H^1_0(\edge)$ (Lemma \ref{Poincare_edge}),
\be \label{Poincare_face}
\norm{\llbracket w \rrbracket}{L^2(\edge)}\le C h_\edge \norm{\nabla_\mesh \llbracket w \rrbracket}{L^2(\edge)^d}. 
\ee
If $\edge \in \edgesint$ with $\mesh_{\edge}= \{\cell, L\}$ then $\llbracket w\rrbracket=0$ at the vertices of $\edge$, and \eqref{Poincare_face} combined with the trace inequality \cite[Lemma 1.46]{DG_DA} therefore give
\begin{align}
\norm{\llbracket w \rrbracket}{L^2(\edge)} & \le C h_\edge( \norm{\nabla_\mesh w_{\lvert \cell}}{L^2(\edge)^d} + \norm{\nabla_\mesh w_{\lvert L}}{L^2(\edge)^d} ) \nonumber\\
& \le C_{\mbox{\scriptsize tr}}h_\edge(h_{\cell}^{-1/2} \norm{\nabla_\mesh w}{L^2(\cell)^d} +h_{L}^{-1/2} \norm{\nabla_\mesh w}{L^2(L)^d} ), \label{norm_jump}
\end{align}
where $C_{\mbox{\scriptsize tr}}$ depends only on $d$ and the mesh regularity parameter $\eta$.
Take $v \in X_{\disc,0}$. Since $\nabla_\disc v$ is continuous at the vertices of elements in $\mesh$ and $\nabla_\disc v$ vanish at vertices along $\partial \O$, choosing $w=\nabla_\disc v$ in \eqref{Poincare_face} and \eqref{norm_jump} yields
$$\norm{\llbracket \nabla_\disc v \rrbracket}{L^2(\edge)^d}
\le C_{\mbox{\scriptsize tr}} h_\edge\Big(h_{\cell}^{-1/2} \norm{\nabla_\mesh (\nabla_\disc v)}{L^2(\cell)^{d \times d}} +h_{L}^{-1/2} \norm{\nabla_\mesh (\nabla_\disc v)}{L^2(L)^{d \times d}} \Big).$$
Recalling the definition \eqref{dg} of the dG norm, the above inequality and the coercivity property of $B$ yield
\begin{align*}
\norm{\nabla_\disc v}{dG}^2\le{}& \norm{\nabla_\mesh (\nabla_\disc v)}{}^2 \\
&+ 2 C_{\mbox{\scriptsize tr}}\sum_{\edge \in \edges}^{}h_\edge\Big(h_{\cell}^{-1} \norm{\nabla_\mesh (\nabla_\disc v)}{L^2(\cell)^{d \times d}}^2 +h_{L}^{-1} \norm{\nabla_\mesh (\nabla_\disc v)}{L^2(L)^{d \times d}}^2 \Big) \\
 \le{}&\norm{\nabla_\mesh (\nabla_\disc v)}{}^2 +C \sum_{\cell \in \mesh}^{}\norm{\nabla_\mesh (\nabla_\disc v)}{L^2(\cell)^{d \times d}}^2\\
\le{}& C \norm{\hessian_{\mesh}(\Pi_\disc v)}{}^2 \le C\varrho^{-2} \norm{\hb_{\mesh}(\Pi_\disc v)}{}^2= C\varrho^{-2} \norm{\hbd v}{}^2.
\end{align*}
Using the fact that $\norm{w}{} \le C \norm{w}{dG}$  whenever $w$ is a broken polynomial on $\mesh$ (see \cite[Theorem 5.3]{DG_DA}), we infer that $\norm{\nabla_\disc v}{} \le C \varrho^{-1}\norm{\hbd v}{}$, which concludes the estimate on $C_\disc^B$.

\smallskip

$\bullet$ \textsc{Consistency:} Consistency follows from the affine property of the family of Adini rectangles. Using \cite[Theorem 3.1.5, Chapter 3]{ciarlet1978finite}, for $\varphi \in H^3(\O)\cap H^2_0(\O)$, we obtain
\begin{align*}
&\inf_{w\in X_{\disc,0}} \norm{\hbd w-\hb \varphi}{} \le Ch\lvert \phi \rvert_{3,\O}\,,\quad
\inf_{w\in X_{\disc,0}} \norm{\nabla_\disc w-\nabla \varphi }{} \le Ch^2\lvert \phi \rvert_{3,\O}\\
&\mbox{ and }\quad \inf_{w\in X_{\disc,0}} \norm{\Pi_\disc w- \varphi}{} \le Ch^3\lvert \phi \rvert_{3,\O},
\end{align*}
which implies $S_\disc^B(\varphi)\le C h \lvert \phi \rvert_{3,\O}$.
\smallskip

$\bullet$ \textsc{Limit-conformity}: for $\xi \in H^2(\O)^{d\times d}$ and $v_\disc \in X_{\disc,0}$,
cellwise integration-by-parts (see Lemma \ref{IBP}) yields
\begin{align*}
\int_{\O}^{}(\hessian:B^{\tau}B\xi)\Pi_\disc v_\disc \d\x&=\sum_{\cell \in \mesh}^{} \int_{\cell}^{}(\hessian:A\xi)\Pi_\disc v_\disc \d\x\\
&=\int_{\O}^{}A\xi:\hessian_\disc v_\disc \d\x-\sum_{\cell \in \mesh}^{}
\int_{\partial \cell}^{}(A\xi n_\cell)\cdot\nabla_\disc v_\disc \d s(\x) \\
& \qquad + \sum_{\cell \in \mesh}^{} \int_{\partial \cell}^{}(\div(A\xi)\cdot n_\cell)\Pi_\disc v_\disc \d s(\x).
\end{align*}
For $\cell \in \mesh$ and $\edge \in \edgescv$, let $n_{\cell,\edge}$ be the unit vector normal to $\edge$ outward to $\cell$. For all $\edge \in \edges$, we choose an orientation (that is, a cell $K$ such that $\edge\in\edges_K$) and we set $n_\edge=n_{\cell,\edge}$. We then set $\llbracket w \rrbracket=w_{|\cell}- w_{|L}$ if $\edge\in\edgesint$ with $\mesh_\edge=\{K,L\}$, and $\llbracket w \rrbracket=w_{|\cell}$ if $\edge\in\edgesext$  with $\mesh_\edge=K$. Then
\be
\begin{aligned}
\int_{\O}&(\hessian:A\xi)\Pi_\disc v_\disc \d\x-\int_{\O}^{}A\xi:\hessian_\disc v_\disc \d\x\\
={}&-\sum_{\edge \in \edges}^{} 
\int_{\edge}^{}(A\xi n_\edge)\cdot \llbracket \nabla_\disc v_\disc \rrbracket \d s(\x) 
+ \sum_{\edge \in \edges}^{} \int_{\edge}^{}(\div(A\xi) \cdot n_\edge)\llbracket \Pi_\disc v_\disc \rrbracket \d s(\x).\label{limit_conformity}
\end{aligned}
\ee
Since $\Pi_\disc v_\disc \in H^1_0(\O)\cap C(\overline{\O})$, $\llbracket \Pi_\disc v_\disc \rrbracket=0$. Let $\Lambda_\cell$ denote the  $Q_1$ interpolation operator associated with the values at the four vertices of $\cell$, and $\Lambda_h$ be the patched interpolator such that $(\Lambda_h)_{|\cell}=\Lambda_\cell$ for all $\cell$. $\Lambda_h(\nabla_\disc v_\disc)$ takes the values of $\nabla_\disc v_\disc$ at the vertices, so it is continuous at internal vertices and vanishes at the boundary vertices. Hence, for any $\edge\in\edges$, $\llbracket\Lambda_h(\nabla_\disc v_\disc)\rrbracket$ vanishes on $\edge$ since it is linear on this edge and vanishes at its vertices.
As a consequence,
\begin{align}
\int_{\O}^{}(\hessian:A\xi)&\Pi_\disc v_\disc \d\x-\int_{\O}^{}A\xi:\hessian_\disc v_\disc \d\x\nonumber\\
&=-\sum_{\edge \in \edges}^{} 
\int_{\edge}^{}(A\xi n_\edge)\cdot \llbracket \nabla_\disc v_\disc-\Lambda_h(\nabla_\disc v_\disc ) \rrbracket \d s(\x) \nonumber \\ 
&=-\sum_{\cell \in \mesh}^{}\sum_{\edge \in \edgescv}^{} 
\int_{\edge}^{}A\xi n_{\cell,\edge}\cdot \Big(\nabla_\disc v_\disc-\Lambda_\cell(\nabla_\disc v_\disc ) \Big) \d s(\x).\label{adini.wd}
\end{align}
Setting $\varphi=A\xi n_{\cell,\edge}$ and $w=\nabla_\disc v_\disc$, a change of variables yields
\be
\int_{\edge\in \edgescv}^{}\varphi\cdot \big(w-\Lambda_\cell(w) \big) \d s(\x)=|\edge|\int_{\hat{\edge}\in \edges_{\hat{\cell}}}^{}\hat{\varphi}\cdot \big(\hat{w}-\Lambda_{\hat{\cell}}(\hat{w} ) \big) \d s(\x),\label{adini.ref.cell}
\ee
where $\hat{\cell}$ is the reference finite element. Let $\edgescv=\{\edge_1^{'},\edge_2^{'},\edge_1^{''},\edge_2^{''} \}$ such that $|\edge_1^{'}|=|\edge_1^{''}|=h_1$ and $|\edge_2^{'}|=|\edge_2^{''}|=h_2$. Let us consider 
\be
\delta_{1,\cell}(\phi,v)= \int_{\edge_1^{'}}^{}\phi \big(v-\Lambda_\cell(v) \big) \d s(\x)-\int_{\edge_1^{''}}^{}\phi \big(v-\Lambda_\cell(v) \big) \d s(\x),\label{adini.delta}
\ee
for $\phi \in H^1(\cell)$ and $v \in \partial_1\mathbb{P}_\cell$. The steps in \cite[Theorem 6.2.3]{ciarlet1978finite} show that $\delta_{1,\cell}(\phi,v) \le Ch |{\phi}|_{1,{\cell}}|{v}|_{1,{\cell}}$. For the sake of completeness, let us briefly recall the argument. Using changes of variables, $\delta_{1,\cell}(\phi,v)=h_1\delta_{1,\hat{\cell}}(\hat{\phi},\hat{v})$. Since $\mathbb{P}_0 \subset Q_1$, which is preserved by $\Lambda_\cell$, for all $\hat{v} \in \mathbb{P}_0$ and $\hat{\phi} \in H^1(\hat{\cell})$ we have  $\delta_{1,\hat{\cell}}(\hat{\phi},\hat{v})=0$ (first polynomial invariance).
Let us now prove that the same relation holds if $\hat{\phi} \in \mathbb{P}_0$ and $\hat{v} \in \partial_1 P_{\hat{\cell}}$. Since $\hat{\phi} \in \mathbb{P}_0$, its value on $\hat{\cell}$ is
a constant, say, equal to $a_0$. Since $\hat{v} \in \partial_1 P_{\hat{\cell}}$ we have
$$\hat{v}=b_0+b_1x_1+b_2x_2+b_3x_1^2+b_4x_1x_2+b_5x_2^2+b_6x_1^2x_2+b_7x_2^3.$$
Taking the values at the four vertices, we get
$$\Lambda_{\hat{\cell}}\hat{v}=b_0+(b_1+b_3)x_1+(b_2+b_5+b_7)x_2+(b_4+b_6)x_1x_2.$$
Assuming without loss of generality that $\edge_1^{'}$ is the line $x_1 = 1$ and $\edge_1^{''}$ is the line $x_1 = 0$, we infer
$$(\hat{v}-\Lambda_{\hat{\cell}}\hat{v})_{|x_1=0}=-(b_5+b_7)x_2+b_5x_2^2+b_7x_2^3,$$
$$(\hat{v}-\Lambda_{\hat{\cell}}\hat{v})_{|x_1=1}=-(b_5+b_7)x_2+b_5x_2^2+b_7x_2^3.$$
The relation $\delta_{1,\hat{\cell}}(\hat{\phi},\hat{v})=0$ (second polynomial invariance) then follows from
\begin{align*}
\int_{\edge_1^{'}}\hat{\phi}(\hat{v}-\Lambda_{\hat{\cell}}\hat{v})\d s(\x)={}&\int_{0}^{1} a_0(-(b_5+b_7)x_2+b_5x_2^2+b_7x_2^3)\dx_2\\
={}&\int_{\edge_1^{''}}\hat{\phi}(\hat{v}-\Lambda_{\hat{\cell}}\hat{v})\d s(\x).
\end{align*}
The bilinear form $\delta_{1,\hat{\cell}}(\hat{\phi},\hat{v})$ is continuous over the space $H^1(\hat{\cell}) \times \partial_1 P_{\hat{\cell}}$ by the trace theorem. Using the bilinear lemma \cite[Theorem 4.2.5]{ciarlet1978finite}, we deduce from the two polynomial invariances the existence of a constant $C$ such that $|\delta_{1,\hat{\cell}}(\hat{\phi},\hat{v})|\le C|\hat{\phi}|_{1,\hat{\cell}}|\hat{v}|_{1,\hat{\cell}}$ for all $\hat{\phi} \in H^1(\hat{\cell}),\, \hat{v} \in \partial_1 P_{\hat{\cell}}$. A direct change of variables shows that
 $$|\hat{\phi}|_{1,\hat{\cell}}\le C |{\phi}|_{1,{\cell}}\quad\mbox{ and }\quad |\hat{v}|_{1,\hat{\cell}}\le C |{v}|_{1,{\cell}}.$$
Since $\delta_{1,\cell}(\phi,v)=h_1\delta_{1,\hat{\cell}}(\hat{\phi},\hat{v})$, we infer $\delta_{1,\cell}(\phi,v) \le Ch |{\phi}|_{1,{\cell}}|{v}|_{1,{\cell}}$. Similarly, $\delta_{2,\cell}(\phi,v) \le Ch |{\phi}|_{1,{\cell}}|{v}|_{1,{\cell}}$ (considering integrals over $\edge_2^{'}$ and $\edge_2^{''}$).
Hence, from \eqref{adini.wd}, \eqref{adini.ref.cell} and \eqref{adini.delta},
\begin{align*}
\Big\lvert\int_{\O}^{}&(\hessian:A\xi)\Pi_\disc v_\disc \d\x-\int_{\O}^{}A\xi:\hessian_\disc v_\disc \d\x\Big\rvert \le C\norm{\xi}{H^2(\O)^{d\times d}}h \norm{\hbd v_\disc}{}.
\end{align*}
The proof of the estimate on $W_\disc^B(\xi)$ is complete.
\end{proof}

\appendix

\section{Technical results} \label{appendix}

	\begin{lemma}[{Poincar\'e} inequality along an edge]{} \label{Poincare_edge}
		Let $\edge$ be an edge of a polygonal cell, $w \in H^1(\edge)$ and assume that $w$ vanish at a point on the edge $\edge \in \edges$. Then there exists $ C > 0 $ such that
			$$\norm{w}{L^2(\edge)} \le  h_{\edge} \norm{\partial w}{L^2(\edge)},$$ where $\partial$ denotes the derivative along the edge and $h_\edge$ is the length of the edge.
	\end{lemma}
	\begin{proof}
 Let $m$ denote the point on the edge $\edge$ which satisfies $w(m) = 0$. For $m < x$, we get
 $$w(x) = w(m) + \int_{m}^{x}\partial w(y)\d y = \int_{m}^{x}\partial w(y)\d y.$$
 A use of Cauchy-Schwarz inequality yields
 $$|w(x)|\le |x-m|^{1/2}\bigg(\int_{m}^{x}|\nabla w|^2\d y\bigg)^{1/2}\le \sqrt{h_\edge}\bigg(\int_\edge|\partial w|^2\d y\bigg)^{1/2}.$$
 	Squaring this yields $|w(x)|^2\le h_\edge\int_\edge|\partial w|^2\d y$ and integrating over the edge concludes the proof.
 \end{proof}

		\begin{lemma}[Integration by parts]\label{IBP}
Let $P$ be a fourth order tensor. For $\xi \in H^2({\O})^{d \times d}$ and $\phi \in H^1(\O)$, we have
			$$\int_\O (\hessian:P\xi)\phi=-\int_\O \nabla\phi\cdot\div(P\xi) + \int_{\partial\O} \div(P\xi\cdot n)\phi.$$
For $\psi \in H^2(\O)$,
				$$\int_\O P\xi:\hessian \psi=-\int_\O \nabla\psi\cdot\div(P\xi) + \int_{\partial\O} (\div(P\xi n))\cdot\nabla \psi.$$
For $\zeta\in H^1(\O)^d$,
\[
\int_\O P\xi:\nabla \zeta= -\int_\O \div(P\xi)\cdot \zeta+
\int_ {\partial\O}(\div(P\xi n))\cdot \zeta.
\]
\end{lemma}
\bibliographystyle{abbrv}
\bibliography{Fourth_order_elliptic}

\begin{thebibliography}{10}

\bibitem{BSBPK84}
S.~Balasundaram and P.~K. Bhattacharyya.
\newblock A mixed finite element method for fourth order elliptic equations
  with variable coefficients.
\newblock {\em Comput. Math. Appl.}, 10(3):245--256, 1984.

\bibitem{BPK85}
P.~K. Bhattacharyya.
\newblock Mixed finite element methods for fourth order elliptic problems with
  variable coefficients.
\newblock In {\em Variational methods in engineering ({S}outhampton, 1985)},
  pages 2.3--2.12. Springer, Berlin, 1985.

\bibitem{BS94}
S.~Brenner and L.~Scott.
\newblock {\em {T}he {M}athematical {T}heory of {F}inite {E}lement {M}ethods}.
\newblock Springer--Verlag, New York, 1994.

\bibitem{CGZ16}
H.~Chen, H.~Guo, Z.~Zhang, and Q.~Zou.
\newblock A ${C}^0$ linear finite element method for two fourth-order
  eigenvalue problems.
\newblock {\em IMA J. Numer. Anal.}, 2016.
\newblock DOI:https://doi.org/10.1093/imanum/drw051.

\bibitem{CWD14}
J.~Chen, D.~Wang, and Q.~Du.
\newblock Linear finite element super-convergence on simplicial meshes.
\newblock {\em Mathematics of Computation}, 83:2161--2185, 2014.

\bibitem{ciarlet1978finite}
P.~G. Ciarlet.
\newblock {\em The finite element method for elliptic problems}.
\newblock Access Online via Elsevier, 1978.

\bibitem{DG_DA}
D.~A. Di~Pietro and A.~Ern.
\newblock {\em Mathematical aspects of discontinuous {G}alerkin methods},
  volume~69 of {\em Math\'ematiques \& Applications (Berlin) [Mathematics \&
  Applications]}.
\newblock Springer, Heidelberg, 2012.

\bibitem{dou_percell_scott}
J.~Douglas, Jr., T.~Dupont, P.~Percell, and R.~Scott.
\newblock A family of {$C^{1}$}\ finite elements with optimal approximation
  properties for various {G}alerkin methods for 2nd and 4th order problems.
\newblock {\em RAIRO Anal. Num\'er.}, 13(3):227--255, 1979.

\bibitem{fvca8-agdm}
J.~Droniou and R.~Eymard.
\newblock The asymmetric gradient discretisation method.
\newblock In {\em Finite volumes for complex applications {VIII}---methods and
  theoretical aspects}, volume 199 of {\em Springer Proc. Math. Stat.}, pages
  311--319. Springer, Cham, 2017.

\bibitem{koala}
J.~Droniou, R.~Eymard, T.~Gallou\"et, C.~Guichard, and R.~Herbin.
\newblock {\em The gradient discretisation method}, volume~82 of {\em
  Mathematics \& Applications}.
\newblock Springer, 2018.
\newblock To appear, \texttt{https://hal.archives-ouvertes.fr/hal-01382358}.

\bibitem{EGH00}
R.~Eymard, T.~Gallou{\"e}t, and R.~Herbin.
\newblock Finite volume methods.
\newblock In P.~G. Ciarlet and J.-L. Lions, editors, {\em Techniques of
  Scientific Computing, Part III}, Handbook of Numerical Analysis, VII, pages
  713--1020. North-Holland, Amsterdam, 2000.

\bibitem{biharmonicFV}
R.~Eymard, T.~Gallou{\"e}t, R.~Herbin, and A.~Linke.
\newblock Finite volume schemes for the biharmonic problem on general meshes.
\newblock {\em Math. Comp.}, 81(280):2019--2048, 2012.

\bibitem{RERH10}
R.~Eymard and R.~Herbin.
\newblock Approximation of the biharmonic problem using piecewise linear finite
  elements.
\newblock {\em C. R. Math. Acad. Sci. Paris}, 348(23-24):1283--1286, 2010.

\bibitem{RF78}
R.~S. Falk.
\newblock Approximation of the biharmonic equation by a mixed finite element
  method.
\newblock {\em SIAM J. Numer. Anal.}, 15(3):556--567, 1978.

\bibitem{benchmark}
R.~Herbin and F.~Hubert.
\newblock Benchmark on discretization schemes for anisotropic diffusion
  problems on general grids.
\newblock In {\em Finite volumes for complex applications {V}}, pages 659--692.
  ISTE, London, 2008.

\bibitem{KLP01}
C.~Kim, R.~Lazarov, J.~Pasciak, and P.~Vassilevski.
\newblock {M}ultiplier {S}paces for the {M}ortar {F}inite {E}lement {M}ethod in
  {T}hree {D}imensions.
\newblock {\em SIAM Journal on Numerical Analysis}, 39:519--538, 2001.

\bibitem{Lam06}
B.~Lamichhane.
\newblock {\em {H}igher {O}rder {M}ortar {F}inite {E}lements with {D}ual
  {L}agrange {M}ultiplier {S}paces and {A}pplications}.
\newblock PhD thesis, Universit\"at Stuttgart, 2006.

\bibitem{Lam11b}
B.~Lamichhane.
\newblock {A} {S}tabilized {M}ixed {F}inite {E}lement {M}ethod for the
  {B}iharmonic {E}quation {B}ased on {B}iorthogonal {S}ystems.
\newblock {\em Journal of Computational and Applied Mathematics},
  235:5188--5197, 2011.

\bibitem{LSW05}
B.~Lamichhane, R.~Stevenson, and B.~Wohlmuth.
\newblock {H}igher {O}rder {M}ortar {F}inite {E}lement {M}ethods in 3{D} with
  {D}ual {L}agrange {M}ultiplier {B}ases.
\newblock {\em Numerische Mathematik}, 102:93--121, 2005.

\bibitem{BL_stab.mixedfem}
B.~P. Lamichhane.
\newblock A stabilized mixed finite element method for the biharmonic equation
  based on biorthogonal systems.
\newblock {\em J. Comput. Appl. Math.}, 235(17):5188--5197, 2011.

\bibitem{BL_FEM}
B.~P. Lamichhane.
\newblock A finite element method for a biharmonic equation based on gradient
  recovery operators.
\newblock {\em BIT}, 54(2):469--484, 2014.

\bibitem{PLPL75}
P.~Lascaux and P.~Lesaint.
\newblock Some nonconforming finite elements for the plate bending problem.
\newblock {\em Rev. Fran\c{c}aise Automat. Informat. Recherche Operationnelle
  S\'er. Rouge Anal. Num\'er.}, 9({\rm R}-1):9--53, 1975.

\bibitem{JL99}
J.~Li.
\newblock Full-order convergence of a mixed finite element method for
  fourth-order elliptic equations.
\newblock {\em J. Math. Anal. Appl.}, 230(2):329--349, 1999.

\bibitem{YLRAKL07}
Y.~Li, R.~An, and K.~Li.
\newblock Some optimal error estimates of biharmonic problem using conforming
  finite element.
\newblock {\em Appl. Math. Comput.}, 194(2):298--308, 2007.

\bibitem{NNBPK96}
N.~Nataraj, P.~K. Bhattacharyya, S.~Balasundaram, and S.~Gopalsamy.
\newblock On a mixed-hybrid finite element method for anisotropic plate bending
  problems.
\newblock {\em Internat. J. Numer. Methods Engrg.}, 39(23):4063--4089, 1996.

\bibitem{percell_ctfem}
P.~Percell.
\newblock On cubic and quartic {C}lough-{T}ocher finite elements.
\newblock {\em SIAM J. Numer. Anal.}, 13(1):100--103, 1976.

\bibitem{powell_sabin}
M.~J.~D. Powell and M.~A. Sabin.
\newblock Piecewise quadratic approximations on triangles.
\newblock {\em ACM Trans. Math. Software}, 3(4):316--325, 1977.

\bibitem{TS80}
T.~Scapolla.
\newblock A mixed finite element method for the biharmonic problem.
\newblock {\em RAIRO Anal. Num\'er.}, 14(1):55--79, 1980.

\bibitem{ZX92}
Z.~H. Xie.
\newblock Error estimate of nonconforming finite element approximation for a
  fourth order elliptic variational inequality.
\newblock {\em Northeast. Math. J.}, 8(3):329--336, 1992.

\bibitem{XZ04}
J.~Xu and Z.~Zhang.
\newblock Analysis of recovery type a posteriori error estimators for mildly
  structured grids.
\newblock {\em Mathematics of Computation}, 73:1139--1152, 2004.

\end{thebibliography}

\end{document}